\documentclass{siamart190516}
\usepackage{amsfonts,amsmath,amssymb}
\usepackage{mathrsfs,mathtools,stmaryrd,wasysym,comment}
\usepackage{enumitem}
  \setlist[itemize]{leftmargin=*}
  \setlist[enumerate]{leftmargin=*}
\usepackage{esint}
\usepackage{tikz}
\usepackage{xspace,mydef}

\DeclareMathAlphabet{\mathpzc}{OT1}{pzc}{m}{it}

\theoremstyle{definition}
\newtheorem{example}{Example}
\newsiamremark{remark}{Remark}

\newcommand{\TheTitle}{Finite element discretization of weighted $\Phi$-Laplace problems}
\newcommand{\ShortTitle}{Finite element discretization of weighted $\Phi$-Laplace problems}
\newcommand{\TheAuthors}{E. Ot\'arola, A. J. Salgado}

\headers{\ShortTitle}{\TheAuthors}

\title{{\TheTitle}\thanks{EO has been partially supported by ANID grant FONDECYT-1220156. AJS has been partially supported by NSF grant DMS-2409918.}}

\author{
Enrique Ot\'arola\thanks{Departamento de Matem\'atica, Universidad T\'ecnica Federico Santa Mar\'ia, Valpara\'iso, Chile. (\email{enrique.otarola@usm.cl}, \url{http://eotarola.mat.utfsm.cl/})}
\and
Abner J. Salgado\thanks{Department of Mathematics, University of Tennessee, Knoxville, TN 37996, USA. (\email{asalgad1@utk.edu}, \url{https://math.utk.edu/people/abner-j-salgado/})}}

\ifpdf
\hypersetup{
  pdftitle={\TheTitle},
  pdfauthor={\TheAuthors}
}
\fi

\date{Draft version of \today.}

\begin{document}

\maketitle
\begin{abstract}
We study the finite element approximation of problems involving the weighted $\Phi$-Laplacian, where $\Phi$ is an $N$-function and the weight belongs to the class $A_\Phi$. In particular, we consider a boundary value problem and an obstacle problem
and derive error estimates in both cases.
The analysis is based on the language of weighted Orlicz and weighted Orlicz--Sobolev spaces.
\end{abstract}



\section{Introduction}
\label{sec:intro}

Let $d \in \mathbb{N}$ with $d \geq 2$, and let $\Omega$ be a bounded polytope in $\mathbb{R}^d$ with Lipschitz boundary $\partial \Omega$. We study finite element methods for approximating the unconstrained and constrained minimization of the following convex energy functional
\begin{equation}
\label{eq:DefOfJ}
  \mathcal{J}(v) \coloneqq \int_{\Omega} \omega(x) \Phi(|\nabla v(x)|)\mathrm{d}x - \int_{\Omega} \omega(x) f(x) v(x) \mathrm{d}x,
\end{equation}
over an appropriate function class that, in particular, satisfies homogeneous Dirichlet boundary conditions. Here $\Phi$ is a suitable convex function, $\omega$ is a weight belonging to a suitable class adapted to $\Phi$, and $f$ is a given forcing term; see Section~\ref{sec:notation_and_prel} for details.

The idea behind the study of minimization problems for energies such as \eqref{eq:DefOfJ} is that in the resulting differential operator, i.e.,
\[
  D \mathcal{J}(v) = - \DIV ( \calbA (x,\nabla v) ),
\]
the vector field $\calbA$ may have a nonstandard growth through $\Phi$ and is not translation invariant due to the dependence on the point $x$. The only simplifying assumption we make is that we consider each of these phenomena separately in the sense that
\begin{equation}
\label{eq:Splitxvec}
  \calbA (x,\bzeta) = \omega(x) \bA(\bzeta), \qquad
  \mathbf{A}: \mathbb{R}^d \rightarrow \mathbb{R}^d,
  \qquad
  \mathbf{A}(\bzeta) = \Phi'(|\bzeta|) \frac{\bzeta}{|\bzeta|}
\end{equation}
for a.e.~$x \in \Omega$ and all $\bzeta \in \mathbb{R}^d$. The prototypical example of the type of problem we have in mind is the so-called weighted $p$-Laplacian: for $p \in (1,\infty)$ and $\kappa \geq 0$, we have
\begin{equation}
\label{eq:p-delta_structure}
 \mathbf{A}(\bzeta) = (\kappa + |\bzeta|)^{p-2} \bzeta \qquad \forall \bzeta \in \mathbb{R}^d.
\end{equation}

The analysis and numerical approximation of quasilinear problems in general and the $p$-Laplacian in particular have a rich history, and any attempt to give a complete bibliographical account is doomed to failure. We mention here only the classical references \cite{MR1192966} and \cite{MR2317830}. In particular in \cite{MR2317830} the authors introduce the language of Orlicz functions to perform an error analysis.
For the approximation of obstacle problems with the $p$-Laplacian as operator, we can refer, for instance, to \cite{MR3852613,MR3745178}.

On the other hand, the study of elliptic problems with nonstandard growth and inhomogeneity has received much attention in recent years, as this seems not only to adopt the difficulties of certain scenarios, but also to unify certain particular structures; see \cite{MR4705246,MR3516828,MR3830720}. It is therefore only natural to investigate the approximation of solutions to such problems. This work can therefore be seen as a step towards the numerical approximation of elliptic problems in Orlicz--Musielak spaces.

Our presentation is as follows. Section~\ref{sec:notation_and_prel} introduces notation and gives an overview of the main ingredients we will use. The goal is to properly describe the functional framework we use: weighted Orlicz and Orlicz--Sobolev spaces. The unconstrained minimization of \eqref{eq:DefOfJ} and its numerical approximation are discussed in Sections~\ref{sec:WPhiLap} and \ref{sec:FEM}, respectively. The analysis hinges on the properties of the Scott-Zhang operator on weighted Orlicz spaces. Finally, an obstacle problem related to \eqref{eq:DefOfJ} and its numerical approximation are studied in Section~\ref{sec:Obstacle}. In this case, the analysis relies on a positivity preserving interpolant and its properties. This operator is analyzed in Section~\ref{sub:Crhn}.

\section{Notation and preliminary remarks}
\label{sec:notation_and_prel}

The first notation we introduce is the relation $A \lesssim B$. This means that $A \leq C B$ for a nonessential constant $C>0$, which can change at each occurrence. The relation $A \gtrsim B$ means $B \lesssim A$, and $A \simeq B$ is the short form for $A \lesssim B \lesssim A$. If it is necessary to explicitly mention a constant $C$, we assume that $C>0$ and that its value can change at each occurrence.

Let $E \subset \mathbb{R}^d$ be a measurable set, and let $|E|$ be the $d$-dimensional Lebesgue measure of $E$. If $0 < |E| < \infty$ and $v \in L^1_{\mathrm{loc}}(\mathbb{R}^d)$, we denote the mean value of $v$ over the set $E$ by
\[
 \langle v \rangle_{E}\coloneqq \fint_E v(x) \diff x = \frac1{|E|} \int_E v(x) \diff x.
\]
Here and in the following, the relation $A \coloneqq B$ indicates that $A$ is \emph{by definition} equal to $B$. The notation $B \eqqcolon A$ means $A \coloneqq B$.

\subsection{$N$-functions, complementary functions, and shifted $N$-functions}

We say that $\Phi:[0,\infty) \rightarrow [0,\infty)$ is an $N$-function if $\Phi$ is differentiable and its derivative $\Phi'$ satisfies the following properties: $\Phi'(0)=0$, $\Phi'(s)>0$ for $s>0$, $\Phi'$ is right-continuous at any point $s \geq 0$, $\Phi'$ is nondecreasing on $[0,\infty)$, and $\lim_{s \rightarrow \infty} \Phi'(s) = \infty$ \cite[Definition 4.2.1]{MR3024912}. We note that every $N$-function is convex \cite[Lemma 4.2.2]{MR3024912}, \cite[page 7]{MR126722}. We say that an $N$-function $\Phi$ satisfies the $\Delta_2$-condition (we shall write $\Phi \in \Delta_2$) if there exists a positive constant $K$ such that $\Phi(2t) \leq K \Phi(t)$ for all $t \geq 0$. The smallest of these constants is denoted by $\Delta_2(\Phi)$ \cite[Definition 4.4.1]{MR3024912}, \cite[page 23]{MR126722}. Since $\Phi(t) \leq \Phi(2t)$, the $\Delta_2$-condition guarantees that $\Phi(t) \simeq \Phi(2t)$ for all $t \geq 0$. More generally, if $\Delta_2(\Phi) < \infty$ and $a > 1$ is fixed, then $\Phi(t) \simeq \Phi(at)$ for all $t \geq 0$ \cite[Exercise 4.5.5]{MR3024912}, \cite[inequality (4.2)]{MR126722}. Finally, we mention that if $\Phi$ is an $N$-function and $\Delta_2(\Phi) < \infty$, then $\Phi(t) \simeq \Phi'(t) \, t$ uniformly in $t \geq 0$ \cite[Proposition 7(2.6b)]{Kreuzer}.

To characterize the behavior of an $N$-function $\Phi$ for very small and very large values of its argument, we introduce the lower and upper indices, respectively, as \cite[Definition 1.1.5]{MR1156767}, \cite[page 71]{MR2797562}
\begin{equation}
\label{eq:DefofIndices}
  i(\Phi) \coloneqq \lim_{\lambda \downarrow 0} \frac1{\log\lambda} \log\left( \sup_{t>0} \frac{\Phi(\lambda t)}{\Phi(t)} \right), 
  \qquad
  I(\Phi) \coloneqq \lim_{\lambda \uparrow \infty} \frac1{\log\lambda} \log\left( \sup_{t>0} \frac{\Phi(\lambda t)}{\Phi(t)} \right).
\end{equation}
It follows from the definition that
$1 \leq i(\Phi) \leq I(\Phi) \leq \infty$. Moreover, if $\Phi \in \Delta_2$, then $1 \leq i(\Phi) \leq I(\Phi) < \infty$ \cite[page 71]{MR2797562}. Finally, we have that for every $\epsilon >0$ there are constants $c_\epsilon$ and $C_\epsilon$ satisfying $0<c_\epsilon \leq C_\epsilon$ such that, for every $s,t \geq 0$ \cite[page 114]{MR1155360},
\begin{equation}
\label{eq:UseofIndices}
  c_\epsilon \min\{ t^{i(\Phi)-\epsilon}, t^{I(\Phi) + \epsilon} \} \Phi(s) \leq \Phi(st) \leq C_\epsilon \max\{ t^{i(\Phi)-\epsilon}, t^{I(\Phi)+\epsilon} \}\Phi(s).
\end{equation}

\subsubsection{Complementary functions}
Let $\Phi$ be an $N$-function. Define
$
(\Phi')^{-1}:[0,\infty) \rightarrow [0,\infty)
$
such that
$
(\Phi')^{-1}(t) \coloneqq \sup \{ s : \Phi'(s) \leq t \}
$
\cite[Definition 4.3.1]{MR3024912}. If $\Phi'$ is continuous and strictly increasing in $[0,\infty)$, then $(\Phi')^{-1}$ is the inverse function of $\Phi'$ and vice-versa \cite[Remark 4.3.3]{MR3024912}. We now define \cite[Definition 4.3.1]{MR3024912}
\[
\Phi^{*}:[0,\infty) \rightarrow [0,\infty),
\qquad
\Phi^{*}(t) \coloneqq \int_0^t (\Phi')^{-1}(s) \mathrm{d}s.
\]
The function $\Phi^{*}$ is called the \emph{complementary function} of $\Phi$. We note that $\Phi^{*}$ is also an $N$-function, $(\Phi^{*})'(t) = (\Phi')^{-1}(t)$ for $t >0$, and  $(\Phi^{*})^{*} = \Phi$.

\begin{example}[the $p$-Laplacian]\rm
  Let $p\in (1,\infty)$ and $\kappa \geq 0$. We define the $N$-functions $\Phi_p:[0,\infty) \rightarrow [0,\infty)$ and $\Phi_{p,\kappa}:[0,\infty) \rightarrow [0,\infty)$ as
  \begin{equation}
  \label{eq:examples}
    \Phi_p(t) \coloneqq \frac{1}{p}t^p,
    \qquad
    \Phi_{p,\kappa}(t) \coloneqq \int_0^t \Phi'_{p,\kappa}(s) \mathrm{d}s,
    \qquad
    \Phi'_{p,\kappa}(t) \coloneqq (\kappa + t)^{p-2}t.
  \end{equation}
  We note that both $\Phi_{p,\kappa}$ and $\Phi_{p}$ satisfy the $\Delta_2$-condition with $\Delta_2(\Phi_{p,\kappa}) \leq C 2^{\max\{ 2,p\}}$ and $\Delta_2(\Phi_p) = 2^p$; see \cite[page 376]{MR2914267}. Note that the first estimate is independent of $\kappa$. The corresponding complementary functions are
  \[
    \Phi_p^{*}(t) = \frac{1}{p'}t^{p'},
    \qquad
    \Phi_{p,\kappa}^{*}(t) \simeq (\kappa^{p-1} + t)^{p'-2}t^2,
    \qquad
    t \in [0,\infty),
    \qquad
    \frac{1}{p} + \frac{1}{p'} = 1.
  \]
  According to \cite[page 376]{MR2914267}, we have that $\Delta_2(\Phi_p^{*}) = 2^{p'}$ and $\Delta_2(\Phi_{p,\kappa}^{*}) \leq C 2^{\max\{ 2,p'\}}$.
\end{example}

\subsubsection{Young's inequalities}

If $(\Phi, \Phi^{*})$ are complementary $N$-functions, then
\begin{equation}
\label{eq:Young's_inequality}
  st \leq \Phi(s) + \Phi^{*}(t), \qquad \forall s,t \geq 0.
\end{equation}
The equality holds if and only if $t = \Phi'(s)$ or $s = (\Phi')^{-1}(t)$ \cite[Theorem 4.3.4]{MR3024912}. This is commonly referred to as \emph{Young's inequality}. We also have the following refined versions: If $\Phi,\Phi^{*} \in \Delta_2$, then for all $\delta > 0$ there exists $C_{\delta}>0$, depending on $\delta$, $\Delta_2(\Phi)$, and $\Delta_2(\Phi^{*})$ such that \cite[Lemma 32]{MR2418205}
\begin{align}
\label{eq:refined_Young's_inequality_1}
  st \leq \delta \Phi(s) + C_{\delta} \Phi^{*}(t) \quad \forall s,t \geq 0,
    \\
\label{eq:refined_Young's_inequality_2}
  s \Phi'(t) + \Phi'(s)t \leq \delta \Phi(s) + C_{\delta} \Phi(t)\quad \forall s,t \geq 0.
\end{align}
Note that inequality \eqref{eq:Young's_inequality} is a generalization of the classical Young's inequality,
whereas \eqref{eq:refined_Young's_inequality_1} is a generalization of ``Young's inequality with $\delta$''.

\subsubsection{Shifted $N$-functions}
Given an $N$-function $\Phi$ with $\Phi,\Phi^* \in \Delta_2$, we introduce the family of \emph{shifted} functions $\{ \Phi_a \}_{a \geq 0}$ as in \cite[Definition 22]{MR2418205}:
\begin{equation}
\label{eq:shifted_N_functions}
  \Phi_a: [0,\infty) \rightarrow [0,\infty),
  \quad
  \Phi_a(t) \coloneqq \int_0^t \Phi_a'(s) \mathrm{d}s,
  \quad
  \Phi_a'(t)\coloneqq \Phi'(a + t) \frac{t}{a+t},
  \,
  t \geq 0.
\end{equation}
For all $a \geq 0$, $\Phi_a$ and $\Phi_a^{*}$ are $N$-functions and $\Phi_a,\Phi_a^{*} \in \Delta_2$. More importantly,
\[
  \sup_{a \geq 0} \left\{ \Delta_2(\Phi_a) \right\}_{a\geq0} \cup \left\{ \Delta_2(\Phi_a^*) \right\}_{a\geq0} \leq C( \Delta_2(\Phi), \Delta_2(\Phi^*) );
\]
see \cite[Lemmas 23 and 27]{MR2418205} and \cite[Lemma 6.1]{MR2317830}.

\begin{example}[the shifted $p$-Laplacian]\rm
  Recall the function $\Phi_{p,\kappa}$ introduced in \eqref{eq:examples}. In this case we have, for $a \geq 0$ and $t \geq 0$,
  \[
    \Phi_{p,\kappa,a}(t) \simeq (\kappa + a + t)^{p-2}t^2,
    \quad
    \Phi_{p,\kappa,a}^{*}(t) \simeq ( (\kappa + a)^{p-1} + t)^{p'-2}t^2.
  \]
  For all $a \geq 0$, $\Delta_2( \Phi_{p,\kappa,a} ) \leq C 2^{\max\{ 2,p \}}$ and $\Delta_2( \Phi^{*}_{p,\kappa,a} ) \leq C 2^{\max\{ 2, p \}}$. Thus, $\{ \Phi_a \}_{a \geq 0}$ and $\{ \Phi^*_a \}_{a \geq 0}$ satisfy the $\Delta_2$-condition uniformly with respect to $a$; see \cite[page 376]{MR2914267}.
\end{example}

\subsection{Orlicz spaces}
\label{sub:Orlicz}
Let $\mu$ be a Borel measure on $\R^d$ that is absolutely continuous with respect to the Lebesgue measure. We denote by $L^0(\Omega)$ the set of all Lebesgue measurable functions in $\Omega$. Let $\Phi$ be an $N$-function. We define $\varrho: L^0(\Omega) \rightarrow \mathbb{R}$ by
\[
  \varrho(v, \Phi) \coloneqq \int_\Omega \Phi(|v(x)|) \mathrm{d}\mu(x).
\]
The function $\varrho$ is a semimodular on $L^0(\Omega)$. Moreover, if $\Phi$ is positive, then $\varrho$ is a modular \cite[Lemma 2.3.10]{MR2790542}. We define the \emph{Orlicz space} \cite[Sec. 3.1, Definition 5]{MR1113700}
\[
  L^\Phi(\mu,\Omega )\coloneqq \left\{ v \in L^0(\Omega): \exists k > 0 : \varrho( k v, \Phi) < \infty  \right\}
\]
equipped with the Luxemburg norm (see \cite[Sec. 3.2, (6) and Theorem 3]{MR1113700})
\begin{equation}
   \| v \|_{L^\Phi(\mu,\Omega)} \coloneqq \inf \left \{  k > 0 : \varrho\left(\frac{v}{k},\Phi \right) \leq 1 \right \}.
 \label{eq:Luxemburg}
\end{equation}
The following is a list of properties of the space $L^\Phi(\mu,\Omega)$:
\begin{itemize}
  \item $L^\Phi(\mu,\Omega)$ is a Banach space \cite[Sec. 3.3, Theorem 10]{MR1113700}.

  \item If $\Phi \in \Delta_2$, then (see \cite[Sec. 3.1, Theorem 2 and Sec. 3.3, Proposition 3]{MR1113700})
  \[
    L^\Phi(\mu,\Omega) = \left\{ v \in L^0(\Omega): \varrho( v, \Phi) < \infty \right\}.
  \]

  \item If $\Phi \in \Delta_2$, then $L^\Phi(\mu,\Omega)$ is separable \cite[Sec. 3.5, Thm. 1 and Sec. 3.4, Cor. 5]{MR1113700}.

  \item If $(\Phi,\Phi^*)$ is a pair of complementary $N$-functions, then we have H\"older's inequality \cite[Sec. 3.3, Proposition 1 and (4)]{MR1113700} (the constant $2$ cannot be omitted):
 \[
    \int_\Omega | v(x) w(x)| \mathrm{d}\mu(x) \leq 2 \| v \|_{L^\Phi(\mu,\Omega)}\| w \|_{L^{\Phi^*}(\mu,\Omega)} \qquad \forall v \in L^\Phi(\mu,\Omega),\ \forall w \in L^{\Phi^*}(\mu,\Omega).
 \]

 \item If $( \Phi,\Phi^* )$ is a pair of complementary $N$-functions and $\Phi, \Phi^* \in \Delta_2$, then $[L^{\Phi}(\mu,\Omega)]^*$ is isomorphic to $L^{\Phi^*}(\mu,\Omega)$ in the following sense \cite[Sec. 4.1, Theorem 7]{MR1113700}:
  \[
    w \in L^{\Phi^*}(\mu,\Omega) \mapsto J_w \in [L^{\Phi}(\mu,\Omega)]^*:
    \quad
    J_w(v) \coloneqq \int_{\Omega} v(x)w(x) \mathrm{d}\mu(x), \ v \in L^{\Phi}(\mu,\Omega),
  \]
  and $\| J_w \|_{L^{\Phi}(\mu,\Omega)^*} \leq 2 \| w \|_{L^{\Phi^*}(\mu,\Omega)}$. $L^{\Phi}(\mu,\Omega)$ is reflexive \cite[Sec. 4.1, Thm. 10]{MR1113700}.
\end{itemize}

For further facts and properties of Orlicz spaces and their generalizations, we refer the reader to \cite{MR126722,MR1113700,MR2790542,MR3931352,MR3889985}.

\subsection{Weights}
\label{sec:A_p_weights}
A weight $\omega$ is a function in $L^1_{\mathrm{loc}}(\R^d)$ such that $\omega(x)>0$ for a.e.~$x \in \mathbb{R}^d$. For a weight $\omega$ and a measurable set $A \subset \mathbb{R}^d$, we set $\omega(A) \coloneqq \int_A \omega(x) \diff x$.

Let $p \in (1,\infty)$. We say that a weight $\omega$ belongs to the Muckenhoupt class $A_p$ if there is a positive constant $C$ such that for every ball $B \subset \R^d$ \cite[Definition 1.2.2]{MR1774162}
\begin{equation}
\label{eq:A_p_condition}
  \left( \fint_B \omega(x) \diff x \right) \left( \fint_B \omega(x)^{-\tfrac1{p-1}} \diff x \right)^{p-1} \leq C.
\end{equation}
The infimum over all such constants $C$ is called the Muckenhoupt characteristic of $\omega$ and is denoted by $[\omega]_{A_p}$. We note that
\[
  \omega \in A_p \iff \omega' \coloneqq \omega^{-\tfrac1{p-1}} \in A_{p'},
 \quad
 [\omega]_{A_p} = [\omega']_{A_{p'}}^{p-1},
 \quad
 p \in (1,\infty),
 \quad
 \frac{1}{p} + \frac{1}{p'} = 1;
\]
see \cite[Remark 1.2.4, item 4]{MR1774162}.
Many useful properties follow from the fact that $\omega \in A_p$. We mention here two of them that will be useful for us in what follows:

\begin{itemize}
  \item \emph{Open ended property}: If $\omega \in A_p$ with $p\in (1,\infty)$, then there is $\delta >0$ such that $\omega \in A_{p-\delta}$; see \cite[Corollary 1.2.17]{MR1774162} and \cite[Corollary 7.6, item (2)]{MR1800316}.
  
  \item \emph{Embedding:} If $1 < p < q < \infty$, then $A_p \subset A_q$; see \cite[Remark 1.2.4, item 3]{MR1774162}.
\end{itemize}

We refer to \cite{MR1774162,MR1800316,MR2305115,MR3243734,MR2797562} for more details on Muckenhoupt weights. 

When dealing with the obstacle problem we restrict ourselves to weights that behave well near $\partial\Omega$. The following definition is inspired by \cite[Definition 2.5]{MR1601373}.

\begin{definition}[$A_p(\Omega)$]
Let $\Omega \subset \R^d$ be bounded and Lipschitz, and let $p \in (1,\infty)$. Given $\omega \in A_p$, we say $\omega \in A_p(\Omega)$ if there is an open $\calG \subset \Omega$, and $\varepsilon,\omega_l>0$ such that
$
\{ x \in \Omega : \dist(x,\partial\Omega) < \varepsilon \} \subset \calG,
$
$
\omega_{|\bar\calG} \in C(\bar\calG),
$
and
$
\inf_{x \in \bar\calG} \omega(x) \geq \omega_l$.
\label{def:A_p(Omega)}
\end{definition}
\subsection{Weighted Orlicz spaces}
\label{sub:WOrlicz}

We have finally reached the point where we can describe the functional framework that we will use in our analysis. Namely, that of \emph{weighted Orlicz spaces}. Let $\omega$ be a weight. Define the measure $\mu$ via $\diff \mu(x) \coloneqq \omega(x) \diff x$. Given an $N$-function $\Phi$, we define the weighted Orlicz space
\[
   L^{\Phi}(\omega,\Omega) \coloneqq \left\{ v \in L^0(\Omega): \exists k > 0 : \ \varrho(k v,\Phi) \coloneqq \int_{\Omega} \omega(x) \Phi(k|v(x)|) \mathrm{d}x < \infty \right\}.
\]
We endow $L^{\Phi}(\omega,\Omega)$ with the Luxemburg norm  \eqref{eq:Luxemburg}. Some properties of this space have already been described in Section~\ref{sub:Orlicz}. Further properties are examined below.

\subsubsection{The class $B_\Phi$}
To define weak derivatives of functions in a weighted Orlicz space, such functions must be distributions. To describe when this is the case, we introduce a class of weights adapted to a given $N$-function $\Phi$. We say that the weight $\omega \in B_\Phi$ if for every ball $B \subset \R^d$ there is a constant $\mu>0$ such that
\begin{equation}
\label{eq:BPhi}
  \frac1\mu \int_B \Phi^*\left( \frac\mu{\omega(x)} \right)\omega(x) \diff x < \infty.
\end{equation}
Condition \eqref{eq:BPhi} is a generalization of the condition that A.~Kufner introduced in \cite[(1.5)]{MR775568} in the context of weighted Sobolev spaces, and it occurs naturally when proving that $L^{\Phi}(\omega,\Omega) \subset L^1_{\mathrm{loc}}(\Omega)$. The proof of this embedding result can be found in \cite[Proposition 3.6 and 3.7]{MR1155360}, where it is shown that $\omega \in B_\Phi$ is not only sufficient but also necessary for the embedding. Because of its simplicity, we give a proof of sufficiency for the sake of completeness.

\begin{proposition}[$L^{\Phi}(\omega,\Omega) \subset L^1_{\mathrm{loc}}(\Omega)$]
\label{prop:LPhiareDist}
Let $\Phi$ be an $N$-function. If $\omega \in B_\Phi$, then every function in $L^{\Phi}(\omega,\Omega)$ is locally integrable in $\Omega$.
\end{proposition}
\begin{proof}

Let $K \subset \Omega$ be compact, let $v \in L^{\Phi}(\omega,\Omega)$, and let $\mu >0$. According to Young's inequality \eqref{eq:Young's_inequality}, we have
\begin{align*}
  \int_{K} |v(x)| \mathrm{d}x &=  \frac1\mu \int_{K} \mu \omega(x)^{-1} |v(x)| \omega(x)  \mathrm{d}x \\
    &\leq \frac1\mu \left[ \int_K \Phi^*\left( \frac\mu{\omega(x)} \right) \omega(x) \diff x + \int_{K} \omega(x) \Phi( |v(x)| ) \mathrm{d}x \right].
\end{align*}
The condition $\omega \in B_\Phi$ controls the first term on the right hand side of the previous inequality, and this shows the result.
\end{proof}

\begin{remark}[class $B_\Phi$]
\label{rem:classBPhi}
Let us consider, for $p \in (1,\infty)$, the function $\Phi_p$ introduced in \eqref{eq:examples}. In this case we have that $\omega \in B_{\Phi_p}$ if
\begin{align*}
    \int_B \Phi_p^*\left( \frac\mu{\omega(x)} \right)\omega(x) \diff x &= 
    \int_B \frac1{p'}\left( \frac\mu{\omega(x)} \right)^{p'}\omega(x) \diff x 
    \simeq \int_B \omega(x)^{-\tfrac1{p-1}} \diff x \leq C.
\end{align*}
In the language of \cite[Definition 1.4]{MR775568}, the finiteness of the last integral for all balls $B \subset \R^d$ means that $\omega \in B_p$. As \cite[Theorem 1.5]{MR775568} shows, $\omega \in B_p$ is sufficient for the inclusion $L^p(\omega,\Omega) \subset L^1_{\mathrm{loc}}(\Omega)$. In this sense, the class $B_\Phi$ stands as the natural extension of the class $B_p$. Note also that $A_p \subset B_p$.
\end{remark}

\subsubsection{The maximal function}
Given a function $v \in L^1_{\mathrm{loc}}(\R^d)$, the Hardy--Littlewood maximal function is defined as
\[
  M[v](x) \coloneqq \sup_{B \ni x } \fint_B |v(y)| \diff y.
\]
The continuity of this maximal function is a very important step in the analysis of a function space. Given an $N$-function $\Phi$ and a weight $\omega$, we say that $\omega \in A_\Phi$ if and only if there is $C>0$ such that for all $\delta >0$ and for all balls $B \subset \R^d$ we have
\begin{equation}
\label{eq:APhi}
  \left[ \fint_B \delta \omega(x) \diff x \right] \Phi'\left( \fint_B (\Phi')^{-1} \left( \frac1{\delta\omega(x)} \right) \diff x \right) \leq C.
\end{equation}

The significance of condition \eqref{eq:APhi} is the content of the following result.

\begin{theorem}[continuity]
\label{thm:MaximalCont}
Let $\Phi$ be an $N$-function such that $\Phi,\Phi^* \in\Delta_2$, and let $\omega$ be a weight. Then, the following statements are equivalent:
\begin{enumerate}
  \item $\omega \in A_\Phi$.
  
  \item $\omega \in A_{i(\Phi)}$.
  
  \item There is a constant $C>0$ such that
  \begin{equation}
  \label{eq:MaximalIsContinuous}
    \int_{\R^d} \Phi( M[v](x) ) \omega(x) \diff x \leq C \int_{\R^d} \Phi(|v(x)|) \omega(x) \diff x
    \quad
    \forall v \in L^\Phi(\omega,\R^d),
  \end{equation}
  where $C$ depends on $\Phi$ only through $\Delta_2(\{ \Phi , \Phi^* \})$ and on $\omega$ only through $[\omega]_{A_{i(\Phi)} }$.
\end{enumerate}
\end{theorem}
\begin{proof}
  This is the content of \cite[Theorem 1]{MR667316} and \cite[Theorem 2.1.1]{MR1156767}. The dependence on the constant $C$ in \eqref{eq:MaximalIsContinuous} is not explicitly stated, but can be inferred from the proof of this result.
\end{proof}

\begin{remark}[$A_\Phi$ is a natural generalization of the classes $A_p$]
  Let us consider the function $\Phi_p$ of \eqref{eq:examples} for $p \in (1,\infty)$. The weight $\omega$ belongs to $A_{\Phi_p}$ if and only if
  \begin{multline*}
    \left[ \fint_B \delta \omega(x) \diff x \right] \Phi_p'\left( \fint_B (\Phi_p')^{-1} \left( \frac1{\delta\omega(x)} \right) \diff x \right) = \\
    \left[\fint_B \omega(x) \diff x \right] \left[ \fint_B \omega(x)^{-\tfrac1{p-1}} \diff x \right]^{p-1} \leq C,
  \end{multline*}
  for all balls $B \subset \mathbb{R}^d$, i.e., $A_{\Phi_p} = A_p$. Thus, $A_\Phi$ is not only the natural generalization of the classes $A_p$, \emph{but also the most general condition} that can be imposed to have a reasonable theory and numerical analysis for PDEs in weighted Orlicz spaces.
  \end{remark}

  \begin{remark}[$A_\Phi \subset B_\Phi$]
  \label{rem:APhisubsetBPhi}
  As we have already mentioned, for any $N$-function $\Phi$, $(\Phi')^{-1} = (\Phi^*)'$. Since $(\Phi^*)'$ is nondecreasing on $[0,\infty)$, $\omega \in A_\Phi$ therefore implies that
  \begin{align*}
    \frac1\mu \int_B \Phi^*\left( \frac\mu{\omega(x)} \right)\omega(x) \diff x &= 
    \frac1\mu \int_B \int_0^{\tfrac\mu{\omega(x)}} (\Phi^*)'(t) \diff t \ \omega(x) \diff x \\
    &\leq
    \frac1\mu \int_B (\Phi^*)'\left( \frac\mu{\omega(x)} \right) \frac\mu{\omega(x)} \omega(x) \diff x \\
    &= \int_B (\Phi^*)'\left( \frac\mu{\omega(x)} \right) \diff x 
    = \int_B (\Phi')^{-1}\left( \frac\mu{\omega(x)} \right) \diff x < \infty,
  \end{align*}
  because the last integral is the argument inside $\Phi'$ in the definition of the class $A_\Phi$. In other words, we have obtained that $A_\Phi \subset B_\Phi$; see Remark~\ref{rem:classBPhi}.
\end{remark}

\subsubsection{Weighed Orlicz--Sobolev spaces}
Let $\Phi$ be an $N$-function, and let $\omega \in A_\Phi \subset B_\Phi$; see Remark \ref{rem:APhisubsetBPhi}. Then, as shown in Proposition~\ref{prop:LPhiareDist}, the elements of the space $L^{\Phi}(\omega,\Omega)$ are distributions and they have distributional derivatives. We can then define the \emph{weighted Orlicz--Sobolev space} \cite[Definition 3.1.1]{MR3889985}
\[
 W^{1,\Phi}(\omega,\Omega) \coloneqq \{ v \in L^{\Phi}(\omega,\Omega): \partial_i v \in L^{\Phi}(\omega,\Omega) \,\, \forall i \in \{1,\ldots,d\} \}
\]
endowed with the norm $\| v \|_{W^{1,\Phi}(\omega,\Omega)} \coloneqq \| v \|_{L^{\Phi}(\omega,\Omega)} + \| \nabla v \|_{L^{\Phi}(\omega,\Omega)}$.
Many of the properties of the classical Sobolev spaces extend to the weighted Orlicz--Sobolev setting:
\begin{enumerate}
 \item $W^{1,\Phi}(\omega,\Omega)$ is a Banach space \cite[Theorem 6.1.4(b)]{MR3931352}.
 \item If $\Phi \in \Delta_2$, then $W^{1,\Phi}(\omega,\Omega)$ is separable \cite[Theorem 6.1.4(c)]{MR3931352}.
 \item If $\Phi , \Phi^{*} \in \Delta_2$, then $W^{1,\Phi}(\omega,\Omega)$ is reflexive \cite[Theorem 6.1.4(d)]{MR3931352}.
\end{enumerate}

We define $W_0^{1,\Phi}(\omega,\Omega)$ as the closure of $C_0^{\infty}(\Omega) \cap W^{1,\Phi}(\omega,\Omega)$ in $W^{1,\Phi}(\omega,\Omega)$ \cite[Definition 6.1.8]{MR3931352}. It follows that this space has the same properties as $W^{1,\Phi}(\omega,\Omega)$; see \cite[Theorem 6.1.9]{MR3931352}. In addition, the following modular Poincar\'e inequality holds.

\begin{proposition}[modular Poincar\'e inequality]
Let $\Phi$ be an $N$-function such that $\Phi, \Phi^* \in \Delta_2$. If $\omega \in A_\Phi$, then
\begin{equation}
\label{eq:WModularPoincare}
  \int_{\Omega} \omega(x) \Phi( |v(x)| ) \mathrm{d}x \lesssim \int_{\Omega} \omega(x) \Phi (| \nabla v(x) |) \mathrm{d}x
  \quad
  \forall v \in W^{1,\Phi}_0(\omega,\Omega).
\end{equation}
\end{proposition}
\begin{proof}
The result follows directly from \cite[Theorem 4.15]{MR2797562}. For the sake of completeness, we provide some context and explanations.
\begin{itemize}
  \item \cite[Theorem 4.15]{MR2797562} begins by letting $\mathcal{B}$ be a Muckenhoupt basis that is $A_{p,\mathcal{B}}$ open. In our context, $\mathcal{B}$ is nothing but the set of all balls $B\subset \R^d$, so that $A_{p,\mathcal{B}} = A_p$, the standard Muckenhoupt class. Thus, the fact that $A_{p,\mathcal{B}}$ is open is the open ended property of the class $A_p$ stated at the end of Section~\ref{sec:A_p_weights}.

  \item The next step is to provide a class $\mathcal{F}$ of pairs $(f,g)$ of nonnegative measurable functions that are not identically zero, so that for some $p_0 \in [1,\infty)$ and all $w_0 \in A_{p_0}$
  \[
    \int_{\R^d} f(x)^{p_0} w_0(x) \diff x \leq C \int_{\R^d} g(x)^{p_0} w_0(x) \diff x.
  \]
  With
  $
    \calF \coloneqq \left\{ (|w|,|\GRAD w|) : w \in C_0^\infty(\Omega) \right\}
  $,
  the inequality is a weighted Poincar\'e inequality in $\Omega$ that holds for any $p_0 \in (1,\infty)$ and all $w_0 \in A_{p_0}$ \cite[Theorem 1.3]{MR0643158}.

  \item In \cite[Theorem 4.15]{MR2797562}, it is assumed that $\Phi$ is an $N$-function such that $1 < i(\Phi) \leq I(\Phi) < \infty$. We note that $1 < i(\Phi) \leq I(\Phi) < \infty$ is equivalent to the requirement that $\Phi, \Phi^* \in \Delta_2$ \cite[page 71]{MR2797562}.
  Let us also note that in \cite{MR2797562} the complementary function of $\Phi$ is denoted by $\bar\Phi$.

  \item In \cite[Theorem 4.15]{MR2797562} it is assumed that $\omega \in A_{i(\Phi)}$. Note that, according to Theorem~\ref{thm:MaximalCont}, $\omega \in A_{i(\Phi)}$ is equivalent to the requirement that $\omega \in A_\Phi$.
\end{itemize}

We can thus conclude that \eqref{eq:WModularPoincare} holds for all $v \in C_0^\infty(\Omega)$. Finally, we argue by density. This concludes the proof.
\end{proof}

\begin{remark}[equivalence]
On $W^{1,\Phi}_0(\omega,\Omega)$, $\|\GRAD w \|_{L^\Phi(\omega,\Omega)}$ defines a norm that is equivalent to $\| w \|_{W^{1,\Phi}(\omega,\Omega)}$ provided that $\Phi,\Phi^* \in \Delta_2$ and $\omega \in A_\Phi$.
\end{remark}

\subsubsection{Weighted Lebesgue and Sobolev spaces}
A specific example of the constructions described above are weighted Lebesgue and Sobolev spaces. Let $p \in (1,\infty)$, let $\Phi_p$ be as in \eqref{eq:examples}, and let $\omega \in A_p$. We set
\[
  W^{1,p}(\omega,\Omega) \coloneqq W^{1,\Phi_p}(\omega,\Omega), \qquad W^{1,p}_0(\omega,\Omega) \coloneqq W^{1,\Phi_p}_0(\omega,\Omega).
\]
On $W^{1,p}_0(\omega,\Omega)$ we have a weighted Poincar\'e inequality: if $D \subset \R^d$ is open, bounded, and Lipschitz, $p \in (1,\infty)$, and $\omega \in A_p$, then \cite[Theorem 1.3]{MR0643158}
\begin{equation}
\label{eq:wPoincare}
  \| w \|_{L^p(\omega,D)} \leq {C_p} \diam (D) \| \GRAD w \|_{L^p(\omega,D)} \quad \forall w\in W^{1,p}_0(\omega,D).
\end{equation}
The constant $C_p$ depends on $\omega$ only through $[\omega]_{A_p}$. The following inequality is also useful for our analysis \cite[Lemmas 3.1 and 4.2]{MR3439216}: There exists $w_D \in \R$ such that
\begin{equation}
\label{eq:wPoincare2}
  \| w - w_{D} \|_{L^p(\omega,D)} \leq {C_p} \diam (D) \| \GRAD w \|_{L^p(\omega,D)} \quad \forall w\in W^{1,p}(\omega,D).
\end{equation}

We also recall a scaled trace inequality. Let $T \subset \R^d$ be a simplex and $F \subset T$ a face of $T$. If $p \in [1,\infty)$ and $w \in W^{1,p}(T)$, then its trace $\tr_{\partial T} w \in L^p(\partial T)$ satisfies
\begin{equation}
  \frac1{|F|}\| \tr_{F} w \|_{L^p(F)}^p \lesssim \frac1{|T|} \| w \|_{L^p(T)}^p + \frac{\diam(T)^p}{|T|} \| \GRAD w \|_{L^p(T)}^p.
\label{eq:ScaledTrace}
\end{equation}

We conclude with a simple embedding result between weighted Orlicz spaces. For a more detailed discussion we refer to \cite[Section 4.5]{MR3024912} and \cite[Theorem 3.8]{MR1155360}.

\begin{proposition}[embedding]
\label{prop:LPhiweightIntoL2weight}
  Let $\Phi$ be an $N$-function such that $\Phi, \Phi^* \in \Delta_2$, and let $\omega \in A_{\Phi}$. If $i(\Phi) > 2$, then we have the embedding $L^\Phi(\omega,\Omega) \hookrightarrow L^2(\omega,\Omega)$. Consequently, and under the same assumptions, $W^{1,\Phi}_0(\omega,\Omega) \hookrightarrow W^{1,2}_0(\omega,\Omega)$.
\end{proposition}
\begin{proof}
  Let $\varrho = \tfrac12( i(\Phi)-2) > 0$ and set $T \geq 1$. Note that, since $2 <  i(\Phi) \leq I(\Phi) < \infty$, we have that $t^{i(\Phi) - \varrho} =\min\{ t^{i(\Phi)-\varrho}, t^{I(\Phi) + \varrho} \}$ for $t\geq T$. In addition, since $i(\Phi)-\varrho > 2$, we also have $t^{i(\Phi) - \varrho} \geq t^2$ for all $t \geq T$. These observations yield
  \begin{equation}
    \frac12 t^2 \leq \frac12\min\{ t^{i(\Phi)-\varrho}, t^{I(\Phi)+\varrho} \} \leq \frac1{2\Phi(1)c_\varrho} \Phi(t)
    \quad
    \forall t \geq T,
    \label{aux:embedding_2}
  \end{equation}
  where we have used \eqref{eq:UseofIndices}. For the rest of the proof, we set $G = (2\Phi(1)c_\varrho)^{-1}>0$. We now let $v \in L^\Phi(\omega,\Omega)$ and define
  $
    \Omega_T := \left\{ x \in \Omega : |v(x)| \geq T \right\}.
  $
  We may then estimate
  \begin{equation}
    \frac12 \int_\Omega |v(x)|^2 \omega(x) \diff x = \frac12 \int_{\Omega_T} |v(x)|^2 \omega(x) \diff x + \frac12 \int_{\Omega\setminus \Omega_T} |v(x)|^2 \omega(x) \diff x.
    \label{aux:embedding_1}
  \end{equation}
  To control the first integral on the right hand side of \eqref{aux:embedding_1}, we use the bound we have just obtained \eqref{aux:embedding_2}:
  \[
    \frac12 \int_{\Omega_T} |v(x)|^2 \omega(x) \diff x \leq G \int_\Omega \Phi(|v(x)|) \omega(x) \diff x.
  \]
  To bound the second integral on the right hand side of \eqref{aux:embedding_1}, we use that $v$ is ``small'' in $\Omega\setminus\Omega_T$, that $\omega$ is a weight, and that $\Omega$ is bounded:
  \[
    \frac12 \int_{\Omega\setminus\Omega_T} |v(x)|^2 \omega(x) \diff x \leq \frac{1}2 T^2 \omega(\Omega).
  \]
  We have therefore arrived at the following estimate
  \[
    \frac12 \int_{\Omega} |v(x)|^2 \omega(x) \diff x \lesssim \int_\Omega \Phi(|v(x)|) \omega(x) \diff x + \omega(\Omega),
  \]
  which proves the embedding.
\end{proof}

\begin{remark}[Sobolev embeddings]
In the context of the last result, it seems natural to expect that $W^{1,\Phi}_0(\omega,\Omega) \hookrightarrow L^q(\omega,\Omega)$ for some $q > i(\Phi)$. After all, when $\Phi = \Phi_p$ so that $i(\Phi) = p$, this is the content of well-known weighted Sobolev embedding results. However, to the best of our knowledge, the only reference dealing with weighted Orlicz-Sobolev embeddings is \cite{MR1190437}, where the definition of weighted Orlicz spaces is different and not equivalent to the one we adopt here.
\end{remark}

\section{The weighted $\Phi$-Laplace problem}
\label{sec:WPhiLap}

We now present the assumptions that allow us to properly formulate and analyze the problem of minimizing \eqref{eq:DefOfJ}. Throughout our discussion, $\Omega \subset \R^d$ is a Lipschitz polytope and $\omega$ is a weight; further assumptions on $\omega$ will be made later. With respect to $\Phi$, we assume the following; compare with \cite[assumption A.1]{MR2914267} and \cite[assumption 2.1]{MR2911397}.

\begin{assumption}[structural assumptions]
\label{ass:equivalence}
Let $\Phi$ be an $N$-function that satisfies the following conditions:
    $\Phi \in C^1([0,\infty)) \cap C^2(0,\infty)$ and
    \begin{equation}
    \label{eq:uniformly_convex}
      \Phi'(t) \simeq t \, \Phi''(t)
    \end{equation}
    uniformly in $t > 0$. The constants hidden in $\simeq$ are called the characteristics of $\Phi$.
\end{assumption}

For $f \in L^{\Phi^{*}}(\omega,\Omega)$, the following variational problem is a necessary and sufficient condition for a minimum of $\calJ$ over $W^{1,\Phi}_0(\omega,\Omega)$: Find $u \in W^{1,\Phi}_0(\omega,\Omega)$ such that
\begin{equation}
  \int_{\Omega} \omega(x) \mathbf{A}(\nabla u(x)) \cdot \nabla v(x) \mathrm{d}x = \int_{\Omega} \omega(x) f(x) v(x)\mathrm{d}x
  \quad
  \forall v \in W_0^{1,\Phi}(\omega,\Omega),
\label{eq:nonlinear_weighted_problem_weak}
\end{equation}
where $\mathbf{A}$ is defined in \eqref{eq:Splitxvec}. We note that, on the basis of Young’s inequality, we have
\[
  \left| \int_{\Omega} \omega(x) f(x) v(x) \mathrm{d}x \right| \leq \int_{\Omega} \omega(x) \Phi(|v(x)|)\mathrm{d}x +  \int_{\Omega} \omega(x) \Phi^{*}(|f(x)|) \mathrm{d}x < \infty,
\]
so that the right hand side of \eqref{eq:nonlinear_weighted_problem_weak} is well-defined. Assume that $\omega \in A_\Phi$. A convex minimization argument guarantees the existence and uniqueness of a weak solution to problem \eqref{eq:nonlinear_weighted_problem_weak}.

\subsection{Some further properties}
The following consequences of Assumption~\ref{ass:equivalence} are important. $\Phi$ is strictly convex in $(0,\infty)$, $\Phi'$ is strictly  monotone increasing in $(0,\infty)$, and $\Phi$ satisfies the $\Delta_2$-condition; see \cite[Remark 42]{Kreuzer} and \cite[page 487]{MR2911397}. Moreover, $\Delta_2(\Phi)$ depends only on the characteristics of $\Phi$. $\Phi^*$ also satisfies Assumption \ref{ass:equivalence}; see \cite[Lemma 25]{MR2418205}. In particular, $\Delta_2(\Phi^*)$ depends only on the characteristics of $\Phi$. Uniformly in $s,t \in \mathbb{R}$, we have the following properties \cite[Lemma 24]{MR2418205}:
\begin{equation}
\label{eq:Phi_pp_p}
\Phi''(|s| + |t|)|s-t| \simeq \Phi'_{|s|}(|s-t|),
\qquad
\Phi''(|s| + |t|)|s-t|^2 \simeq \Phi_{|s|}(|s-t|).
\end{equation}
The following shifted version of Young's inequality holds: for all $\delta>0$ there exists a positive constant $C_{\delta}$ such that \cite[Lemma 32]{MR2418205}
\begin{equation}
 \label{eq:shifted_Young_inequality}
 s \Phi_a'(t) + \Phi_a'(s)t \leq \delta \Phi_a(s) + C_{\delta} \Phi_a(t), \qquad \forall s,t,a \geq 0.
\end{equation}
\subsection{The operator $\mathbf{A}$}
\label{sec:nonlinear_operator_A}
Let $\Phi$ be an $N$-function that satisfies Assumption~\ref{ass:equivalence}. We recall the definition of the vector field $\mathbf{A}$ in \eqref{eq:Splitxvec}, namely,
\begin{equation}
\label{eq:mathbfA}
\mathbf{A}: \mathbb{R}^d \rightarrow \mathbb{R}^d,
\qquad
\mathbf{A}(\bzeta) \coloneqq \Phi'(|\bzeta|) \frac{\bzeta}{|\bzeta|} 
\quad \forall \bzeta \in \mathbb{R}^d.
\end{equation}

\begin{example}[the $p$-Laplacian]\rm
  Let $p \in (1,\infty)$, and let $\kappa \geq 0$. Recall the function $\Phi_{p,\kappa}$ introduced in \eqref{eq:examples}. Note that $\Phi_{p,\kappa}$ satisfies the conditions in Assumption \ref{ass:equivalence} and that the characteristics of $\Phi_{p,\kappa}$ do not depend on $\kappa$. In fact, we have
  \[
   \min\{1,p-1\} (\kappa + t )^{p-2} \leq \Phi_{p,\kappa}''(t) \leq \max\{1,p-1\} (\kappa + t )^{p-2}
  \]
  for every $t>0$ \cite[page 376]{MR2914267}. In this scenario
  \[
    \mathbf{A}(\bzeta) = (\kappa + |\bzeta|)^{p-2} \bzeta, \quad \bzeta \in \R^d.
  \]
\end{example}

\begin{remark}[$\Phi$-monotonicity and $\Phi$-growth]
It is shown in \cite[Lemma 21]{MR2418205} that there exist positive constants $C_0$ and $C_1$ such that
\begin{align}
\label{eq:nonlinear_operator_1}
  (\mathbf{A}(\bzeta) - \mathbf{A}(\beeta)) \cdot ( \bzeta  - \beeta ) \geq C_0 \Phi''(|\bzeta| + |\beeta|) | \bzeta  - \beeta |^2
  \quad
  \forall \bzeta, \beeta \in \mathbb{R}^d,
  \\
\label{eq:nonlinear_operator_2}
  |\mathbf{A}(\bzeta) - \mathbf{A}(\beeta)| \leq C_1 \Phi''(|\bzeta| + |\beeta|) | \bzeta  - \beeta |
  \quad
  \forall \bzeta, \beeta \in \mathbb{R}^d.
 \end{align}
Here, $C_0$ and $C_1$ depend only on $\Delta_2(\Phi)$, $\Delta_2(\Phi^{*})$, and the characteristics of $\Phi$.
\end{remark}

Let us now define the $N$-function $\Psi :[0,\infty) \to [0,\infty)$ by
\[
  \Psi(0) = 0, \qquad \qquad \Psi'(t) \coloneqq \sqrt{t \Phi'(t)} , \quad t>0.
\]
In \cite[Lemma 25]{MR2418205} it is shown that the $N$-functions $\Psi$ and $\Psi^*$ also satisfy Assumption \ref{ass:equivalence} and that $\Psi''(t) \simeq \sqrt{ \Phi''(t) }$ uniformly in $t>0$. With the function $\Psi$ at hand, we introduce the vector field $\mathbf{V}$ as follows:
\begin{equation}
\label{eq:mathbfV}
\mathbf{V} : \mathbb{R}^d \rightarrow \mathbb{R}^d,
\qquad
\mathbf{V}(\bzeta) \coloneqq \Psi'(|\bzeta|) \frac{\bzeta}{|\bzeta|} \quad \forall \bzeta \in \mathbb{R}^d.
\end{equation}
An application of \cite[Lemma 21]{MR2418205} shows that the bounds \eqref{eq:nonlinear_operator_1} and \eqref{eq:nonlinear_operator_2} hold if $\mathbf{A}$ and $\Phi$ are replaced by $\mathbf{V}$ and $\Psi$, respectively.

\begin{example}[the $p$-Laplacian]\rm
  Let $p \in (1,\infty)$, and let $\kappa \geq 0$. The following functions correspond to those defined in \eqref{eq:examples}:
  \begin{equation*}
   \label{eq:examples_2}
     \Psi_p'(t) = t^{\frac{p}{2}},
     \quad
     \mathbf{V}_p(\bzeta) = |\bzeta|^{\frac{p-2}{2}} \bzeta,
     \quad
     \Psi_{p,\kappa}'(t) = (\kappa + t)^{\frac{p-2}{2}} t ,
     \quad
     \mathbf{V}_{p,\kappa}(\bzeta) = (\kappa + |\bzeta|)^{\frac{p-2}{2}} \bzeta.
  \end{equation*}
\end{example}

The following result expresses the relationship between $\mathbf{A}$, $\mathbf{V}$, and $\{ \Phi_a \}_{a \geq 0}$.

\begin{proposition}[equivalences]
Let $\mathbf{A}$ and $\mathbf{V}$ be defined as in \eqref{eq:mathbfA} and \eqref{eq:mathbfV}, respectively. Then, for all $\bzeta, \beeta \in \mathbb{R}^d$, we have 
\begin{align}
\label{eq:fundamental_1}
  (\mathbf{A}(\bzeta) - \mathbf{A}(\beeta)) \cdot (\bzeta - \beeta) & \simeq | \mathbf{V}(\bzeta) - \mathbf{V}(\beeta) |^2
\\
\label{eq:fundamental_2}
  & \simeq \Phi_{|\bzeta|}(|\bzeta - \beeta|)
  \\
\label{eq:fundamental_2_new}
  & \simeq \Phi_{|\beeta|}(|\bzeta - \beeta|)
  \\
\label{eq:fundamental_3}
  & \simeq | \bzeta - \beeta |^2\Phi''(|\bzeta| + |\beeta|).
\end{align}
Moreover, for all $\bzeta, \beeta \in \mathbb{R}^d$, we have
\begin{equation}
 \label{eq:fundamental_4}
 |\mathbf{A}(\bzeta) - \mathbf{A}(\beeta)| \simeq \Phi'_{|\bzeta|}(|\bzeta - \beeta|).
 \end{equation}
The constants hidden in $\simeq$ only depend on the characteristics of $\Phi$.
\end{proposition}
\begin{proof}
 A proof of the estimates \eqref{eq:fundamental_1}--\eqref{eq:fundamental_3} can be found in \cite[Lemma 3]{MR2418205} while a proof for \eqref{eq:fundamental_4} can be found in \cite[Corollary 64]{Kreuzer}.
\end{proof}

\section{Finite element discretization of the $\Phi$-Laplace problem}
\label{sec:FEM}
Now that we have established the well-posedness of \eqref{eq:nonlinear_weighted_problem_weak}, we proceed with its approximation. Since we have assumed that $\Omega$ is a polytope, it can be meshed exactly. Denote by $\T = \{ T \}$ a conforming partition, or mesh, of $\bar \Omega$ into closed simplices $T$ of size $h_T = \mathrm{diam}(T)$. Here, $h = \max\{ h_T: T \in \T \}$. For $T \in \T$, we define $S_T$ as the set of all elements in $\T$ that share at least one vertex with $T$. We denote by $\Tr = \{ \T \}_{h>0}$ a collection of conforming meshes that are refinements of an initial mesh $\mathscr{T}_0$. We assume that $\Tr$ satisfies the so-called shape regularity condition \cite[Definition 1.107]{MR2050138}: there exits a constant $\sigma > 1$ such that
\[
  \sup_{\T \in \Tr} \max \{ \sigma_T: T \in \T\} \leq \sigma,
\]
where $\sigma_T \coloneqq h_T / \rho_T$ is the shape coefficient of $T$ and $\rho_T$ is the diameter of the sphere inscribed in $T$. With this setting, given $\T \in \Tr$, we define the finite element spaces
\[
  \mathbb{W}_h \coloneqq \{ v_h \in W^{1,\infty}(\Omega) : \ v_{h|T} \in \mathbb{P}_1(T) \,\, \forall T \in \T \},
  \qquad
  \mathbb{V}_h \coloneqq \mathbb{W}_h \cap W^{1,\infty}_0(\Omega) .
\]

The finite element approximation of \eqref{eq:nonlinear_weighted_problem_weak} is:
Find $u_h \in \mathbb{V}_h$ such that
\begin{equation}
\label{eq:nonlinear_weighted_problem_discrete}
\int_{\Omega} \omega(x) \mathbf{A}(\nabla u_h(x)) \cdot \nabla v_h(x) \mathrm{d}x = \int_{\Omega} \omega(x) f(x) v_h(x) \mathrm{d}x
\quad
\forall v_h \in \mathbb{V}_h.
\end{equation}
A convex minimization argument yields the existence and uniqueness of $u_h$.

\subsection{A best approximation result}

We present a best approximation result for \eqref{eq:nonlinear_weighted_problem_discrete}. To derive it, we first establish the following Galerkin orthogonality property:
\begin{equation}
 \label{eq:Galerkin_ortogonality}
 \int_{\Omega} \omega(x) \left( \mathbf{A}(\nabla u(x)) - \mathbf{A}(\nabla u_h(x)) \right) \cdot \nabla v_h(x) \mathrm{d}x = 0 \quad \forall v_h \in \mathbb{V}_h.
\end{equation}
This is the first step towards proving \emph{best approximation} for finite element solutions.

\begin{theorem}[best approximation]
\label{thm:CeaLin}
Let $u \in W_0^{1,\Phi}(\omega,\Omega)$ solve  \eqref{eq:nonlinear_weighted_problem_weak}, and let $u_h \in \mathbb{V}_h$ be its finite element approximation defined as the solution to \eqref{eq:nonlinear_weighted_problem_discrete}. Then,
\begin{equation}
\| \mathbf{V}(\nabla u) - \mathbf{V}(\nabla u_h) \|_{L^2(\omega,\Omega)}
\lesssim \inf_{v_h \in \mathbb{V}_h}
\| \mathbf{V}(\nabla u) - \mathbf{V}(\nabla v_h) \|_{L^2(\omega,\Omega)}.
\label{eq:best_approximation}
\end{equation}
\end{theorem}
\begin{proof}
The derivation of the best approximation property \eqref{eq:best_approximation} modifies the arguments in \cite[Lemma 5.2]{MR2317830} to account for our weighted setting. To proceed, we define $e = u-u_h$. Moreover, for an arbitrary but fixed $v_h \in \mathbb{V}_h$, we set $z = u - v_h$. We begin by using the equivalence \eqref{eq:fundamental_1} and Galerkin orthogonality to obtain
\begin{align*}
 \mathfrak{I} &\coloneqq \int_{\Omega} \omega(x) | \mathbf{V}(\nabla u(x)) - \mathbf{V}(\nabla u_h(x)) |^2 \mathrm{d}x \\
   &\simeq \int_{\Omega} \omega(x) ( \mathbf{A}(\nabla u(x)) - \mathbf{A}(\nabla u_h(x))) \cdot \nabla e(x) \mathrm{d}x \\
   &=  \int_{\Omega} \omega(x) ( \mathbf{A}(\nabla u(x)) - \mathbf{A}(\nabla u_h(x))) \cdot \nabla z(x) \mathrm{d}x.
\end{align*}
We now use the equivalences \eqref{eq:fundamental_3} and \eqref{eq:Phi_pp_p} to derive
\begin{align*}
 \mathfrak{I}
 & \simeq \int_{\Omega} \omega(x) \Phi''(|\nabla u(x)| + |\nabla u_h(x)|)|\nabla e(x)||\nabla z(x)|\mathrm{d}x
 \\
 & \simeq \int_{\Omega} \omega(x) \Phi'_{|\nabla u(x)|}(|\nabla e(x)|)|\nabla z(x)|\mathrm{d}x.
\end{align*}
An application of the shifted version of Young's inequality \eqref{eq:shifted_Young_inequality} shows that for all $\delta >0$ there exists a positive constant $C_{\delta}$ so that
\begin{align*}
  \mathfrak{I} & \leq
  C \delta \int_{\Omega} \omega(x) \Phi_{|\nabla u(x)|}(|\nabla e(x)|)\mathrm{d}x
  +
  C_{\delta} \int_{\Omega} \omega(x) \Phi_{|\nabla u(x)|}(|\nabla z(x)|)\mathrm{d}x
  \\
  & \leq
  C \delta \int_{\Omega} \omega(x) | \mathbf{V}(\nabla u(x)) - \mathbf{V}(\nabla u_h(x)) |^2 \mathrm{d}x
  +
  C_{\delta} \int_{\Omega} \omega(x) \Phi_{|\nabla u(x)|}(|\nabla z(x)|)\mathrm{d}x,
\end{align*}
where to obtain the last inequality we used the equivalence \eqref{eq:fundamental_2}. Using the definition of $\mathfrak{I}$ and the equivalence \eqref{eq:fundamental_2} once again we obtain
\[
  \mathfrak{I}
  \leq
  C \delta \mathfrak{I} + C C_{\delta} \| \mathbf{V}(\nabla u) - \mathbf{V}(\nabla v_h) \|_{L^2(\omega,\Omega)}^2.
\]
The term $C\delta \mathfrak{I}$ appearing in the previous inequality can be absorbed on the left hand side if $\delta$ is chosen carefully. This gives us the estimate \eqref{eq:best_approximation} and concludes the proof.
\end{proof}

\subsection{Stability estimates for an interpolation operator}
In this section, we derive stability estimates in weighted Orlicz--Sobolev spaces for a suitable interpolation operator $\Pi_h$. As in \cite{MR2317830}, we assume that such an interpolation operator
\[
 \Pi_h: W^{1,1}(\Omega) \rightarrow \mathbb{W}_h
\]
satisfies the following two properties:
\begin{enumerate}
 \item \emph{Projection:} $\Pi_h$ is a projection, i.e., $\Pi_h w_h = w_h$ for all $w_h \in \mathbb{W}_h$.
 \item \emph{Stability:} Let $v \in W^{1,1}(\Omega)$. Then, for each $T \in \T$ we have the bound
 \begin{equation}
 \label{eq:stability_SZ}
  \| \GRAD \Pi_h v \|_{L^{1}(T)} \lesssim \| \GRAD v \|_{L^{1}(S_T)}.
 \end{equation}
 Recall that $S_T$ is the set of elements in $\T$ that share at least one vertex with $T$.
\end{enumerate}

An example of an interpolation operator $\Pi_h$ that satisfies these two properties is the so-called Scott--Zhang operator, which was introduced in \cite{MR1011446}. The construction of $\Pi_h$ is such that it preserves homogeneous boundary conditions, i.e., $v = 0$ on $\partial \Omega$ implies $\Pi_h v = 0$ on $\partial \Omega$ \cite[(2.17)]{MR1011446}. Moreover, $\Pi_h$ is a projection from $W^{1,1}(\Omega)$ to $\mathbb{W}_h$ with the property that $W_0^{1,1}(\Omega)$ is mapped to $\mathbb{V}_h$ \cite[Theorem 2.1]{MR1011446}. The Scott--Zhang operator also satisfies the stability bound \eqref{eq:stability_SZ}. We refer the reader to \cite[Theorem 3.1]{MR1011446} for a proof of this result and mention that in this case the hidden constant in \eqref{eq:stability_SZ} depends only on the shape-regularity coefficient $\sigma$.

Inspired by the unweighted case discussed in \cite[Theorem 4.5]{MR2317830}, we present a stability result in weighted Orlicz--Sobolev spaces $W^{1,\Phi}(\omega,\Omega)$.

\begin{theorem}[weighted stability]
\label{thm:WStabSZ}
Let $\Phi$ be an $N$-function that satisfies Assumption \ref{ass:equivalence}, let $\omega \in A_\Phi$, and let $v \in W^{1,\Phi}(\omega,\Omega)$. Then, for $T \in \T$, we have
\begin{equation}
 \fint_{T} \omega(x) \Phi( | \nabla \Pi_h v(x)| ) \mathrm{d}x \lesssim \fint_{S_T} \omega(x) \Phi( |\nabla v(x)| ) \mathrm{d}x,
\end{equation}
where the hidden constant depends on $\sigma$, $[\omega]_{A_{i(\Phi)}}$, and the characteristics of $\Phi$.
\label{thm:stability_SZ}
\end{theorem}
\begin{proof}
Let $T \in \T$ and let $x \in T$. We begin the proof by noting that
\[
 | \nabla \Pi_h v(x)| \lesssim \fint_{T} | \nabla \Pi_h v(z)| \mathrm{d}z \leq C \fint_{S_T} | \nabla v(z)| \mathrm{d}z.
\]
The first bound follows from $(\nabla \Pi_h v)_{|T} \in \mathbb{P}_0(T)$, while the second follows from the stability of the interpolation operator $\Pi_h$ in \eqref{eq:stability_SZ}. Note that $C$ depends only on the shape-regularity coefficient $\sigma$. We can therefore rely on the monotonicity of $\Phi$ and the $\Delta_2$-condition, which holds because $\Phi$ satisfies Assumption \ref{ass:equivalence}, to obtain
\begin{align*}
  \fint_{T} \omega(x) \Phi( | \nabla \Pi_h v(x)| ) \mathrm{d}x 
  & \leq
  \fint_{T} \omega(x) \Phi \left( C \fint_{S_T}  | \nabla v(z)| \mathrm{d}z \right) \mathrm{d}x \\
  & \lesssim
  \fint_T \omega(x) \Phi \left( \fint_{S_T} | \nabla v(z)| \mathrm{d}z \right) \mathrm{d}x \\
  &=
    \fint_T \omega(x) \Phi \left( \fint_{S_T} | \nabla v(z) \chi_{S_T}(z)| \diff z \right) \diff x
   \\
  &\lesssim
    \fint_T \omega(x) \Phi \left( M[ \nabla v \chi_{S_T}] (x)\right) \diff x.
\end{align*}
In the last step, we used shape regularity to replace the average on the patch $S_T$ by the average over a ball $B$, which is such that $x \in B$, $S_T \subset B$, and $|S_T| \approx |B|$. The average over the ball $B$ is therefore bounded by the maximal function.
Since $\omega \in A_\Phi$, we then invoke Theorem~\ref{thm:MaximalCont} to obtain
\begin{align*}
  \fint_{T} \omega(x) \Phi( | \nabla \Pi_h v(x)| ) \mathrm{d}x &\lesssim
  \frac1{|T|} \int_{\R^d} \omega(x) \Phi \left( M[ \nabla v \chi_{S_T}](x) \right) \diff x \\
  &\lesssim
  \frac1{|T|} \int_{\R^d} \omega(x) \Phi \left( |\nabla v(x) \chi_{S_T}(x)| \right) \diff x \\
  &\lesssim \frac1{|S_T|} \int_{S_T} \omega(x) \Phi \left( |\nabla v  (x)| \right) \diff x
  = \fint_{S_T} \omega(x) \Phi \left( |\nabla v  (x)| \right) \diff x.
\end{align*}
Note that in the course of our derivations all constants depend only on $\sigma$, $[\omega]_{A_{i(\Phi)}}$, and the characteristics of $\Phi$. This concludes the proof.
\end{proof}

We present the result of Theorem \ref{thm:stability_SZ} in the context of a family of shifted functions.

\begin{corollary}[weighted stability]
\label{cor:WStabSZ}
Let $\Phi$ be an $N$-function satisfying Assumption \ref{ass:equivalence}, and let $\{ \Phi_a \}_{a \geq 0}$ be the family of shifted $N$-functions defined in \eqref{eq:shifted_N_functions} associated with $\Phi$. If $\omega \in A_{\Phi}$, then we have the following bound uniformly in $a \geq 0$ and $T \in \T$:
\begin{equation}
 \fint_{T} \omega(x) \Phi_a ( | \nabla \Pi_h v(x)| ) \mathrm{d}x \lesssim \fint_{S_T} \omega(x) \Phi_a( |\nabla v(x)| ) \mathrm{d}x,
\end{equation}
where the hidden constant depends on $\sigma$, $[\omega]_{A_{i(\Phi)}}$, and the characteristics of $\Phi$.
\label{co:stability_SZ}
\end{corollary}
\begin{proof}
The desired bound results from the application of the stability bound of Theorem \ref{thm:stability_SZ} to the family of shifted $N$-functions $\{ \Phi_a \}_{a \geq 0}$ in conjunction with the fact that such a family satisfies the $\Delta_2$-condition uniformly in $a \geq 0$ with a constant that depends only on $\Delta_2(\Phi)$; see \cite[Lemma 23]{MR2418205} and \cite[Lemma 6.1]{MR2317830}.
\end{proof}

\subsection{A quasi-best approximation result for the Scott-Zhang operator}

In this section, we prove that the Scott--Zhang interpolation operator $\Pi_h$ satisfies the following local quasi-best approximation property. In doing so, we adapt the proof of \cite[Theorem 5.7]{MR2317830} to our weighted setting.

\begin{theorem}[local approximation]
Let $\Phi$ be an $N$-function satisfying Assumption \ref{ass:equivalence}, let $\omega \in A_{\Phi}$, and let $v \in W^{1,\Phi}(\omega,\Omega)$. Then, for each $T \in \T$, we have
\begin{equation}
 \fint_{T} \omega(x) |\mathbf{V}( \nabla v(x) ) - \mathbf{V}( \nabla \Pi_h v(x) )|^2 \mathrm{d}x
 \lesssim
 \inf_{\mathbf{q} \in \mathbb{R}^d} \fint_{S_T} \omega(x) |\mathbf{V}( \nabla v(x) ) - \mathbf{V}( \mathbf{q} )|^2 \mathrm{d}x,
 \label{eq:local_approximation}
\end{equation}
where the hidden constant depends on $\sigma$, $[\omega]_{A_{i(\Phi)}}$, and the characteristics of $\Phi$. In particular, if $\nabla \mathbf{V}( \nabla v)$ belongs to $L^2(\omega,\Omega)$, then we have the bound
\begin{equation}
 \fint_{T} \omega(x) |\mathbf{V}( \nabla v(x) ) - \mathbf{V}( \nabla \Pi_h v(x) )|^2 \mathrm{d}x
 \lesssim
 h_T^2 \fint_{S_T} \omega(x) |\nabla \mathbf{V}(  \nabla v )(x)|^2 \mathrm{d}x,
 \label{eq:local_approximation_corollary}
\end{equation}
where the hidden constant depends on $\sigma$, $[\omega]_{A_{i(\Phi)}}$, and the characteristics of $\Phi$.
\label{thm:local_best_SZ}
\end{theorem}
\begin{proof}
We begin by noting that the equivalence \eqref{eq:fundamental_2_new} allows us to obtain
\[
 \int_{\Omega} \omega(x) |\mathbf{V}(  \nabla v(x) )|^2 \mathrm{d}x \leq C \int_{\Omega} \omega(x) \Phi(|\nabla v(x)|) \mathrm{d}x < \infty
\]
because $v \in W^{1,\Phi}(\omega,\Omega)$. This implies that $\mathbf{V}( \nabla v ) \in L^2(\omega,\Omega)$.

We now let $\mathbf{q} \in \mathbb{R}^d$ be arbitrary and apply the triangle inequality to obtain
\begin{multline*}
 \fint_{T} \omega(x) |\mathbf{V}(  \nabla v(x) ) - \mathbf{V}( \nabla \Pi_h v(x) )|^2 \mathrm{d}x
 \lesssim
 \fint_{T} \omega(x) |\mathbf{V}(  \nabla v(x) ) - \mathbf{V}( \mathbf{q} )|^2 \mathrm{d}x
 \\
 +
 \fint_{T} \omega(x) |\mathbf{V}( \mathbf{q} ) - \mathbf{V}( \nabla \Pi_h v(x) )|^2 \mathrm{d}x \eqqcolon \mathbf{I}_T + \mathbf{II}_T.
\end{multline*}
Let $\mathfrak{p} \in \mathbb{P}_1(S_T)$ be such that $\nabla \mathfrak{p} = \mathbf{q}$. Note that $\mathbf{q} = \nabla \mathfrak{p} = \nabla \Pi_h \mathfrak{p}$ because $\Pi_h \mathfrak{p} = \mathfrak{p}$. We again use the equivalence \eqref{eq:fundamental_2} and the previous observation to bound the term $\mathbf{II}_T$ as follows:
\begin{multline*}
 \fint_{T} \omega(x) |\mathbf{V}( \mathbf{q} ) - \mathbf{V}( \nabla \Pi_h v(x) )|^2 \mathrm{d}x
 \lesssim
 \fint_{T} \omega(x) \Phi_{|\mathbf{q}|}(|\mathbf{q} - \nabla \Pi_h v(x)|)  \mathrm{d}x
 \\
 = C
 \fint_{T} \omega(x) \Phi_{|\mathbf{q}|}(|\nabla \Pi_h \mathfrak{p}(x) - \nabla \Pi_h v(x)|)  \mathrm{d}x  \\
 = C \fint_{T} \omega(x) \Phi_{|\mathbf{q}|}(|\nabla \Pi_h ( \mathfrak{p} - v)(x)|)  \mathrm{d}x.
\end{multline*}
We now apply Corollary \ref{co:stability_SZ} with $a = | \mathbf{q} |$ to the function $\mathfrak{p} - v$ and arrive at
\[
  \fint_{T} \omega(x) |\mathbf{V}( \mathbf{q} ) - \mathbf{V}( \nabla \Pi_h v(x) )|^2 \mathrm{d}x
  \lesssim
  \fint_{S_T} \omega(x) \Phi_{|\mathbf{q}|}(| \mathbf{q} - \nabla v(x)|)  \mathrm{d}x.
\]
As a final step, we apply the equivalence \eqref{eq:fundamental_2} once again to derive a bound for $\mathbf{II}_T$:
\[
  \mathbf{II}_T
  \lesssim
  \fint_{S_T} \omega(x) \Phi_{|\mathbf{q}|}(| \mathbf{q} - \nabla v(x)|)  \mathrm{d}x
  \leq C
  \fint_{S_T} \omega(x) |\mathbf{V}( \mathbf{q} ) - \mathbf{V}( \nabla v(x) )|^2 \mathrm{d}x.
\]

We now bound $\mathbf{I}_T$. To do this, we use that $T \subset S_T$ with $|T| \simeq |S_T|$ to obtain
\[
  \mathbf{I}_T
  \leq C
  \fint_{S_T} \omega(x) |\mathbf{V}( \mathbf{q} ) - \mathbf{V}( \nabla v(x) )|^2 \mathrm{d}x.
\]

Combining the bounds for $\mathbf{I}_T$ and $\mathbf{II}_T$ we conclude that, for every $\mathbf{q} \in \mathbb{R}^d$,
\begin{equation}
  \fint_{T} \omega(x) |\mathbf{V}( \nabla v(x) ) - \mathbf{V}( \nabla \Pi_h v(x) )|^2 \mathrm{d}x
  \leq C
  \fint_{S_T} \omega(x) |\mathbf{V}( \mathbf{q} ) - \mathbf{V}( \nabla v(x) )|^2 \mathrm{d}x .
\label{eq:local_approximation_with_q}
\end{equation}
This estimate immediately yields \eqref{eq:local_approximation}.

After we have derived the local approximation result \eqref{eq:local_approximation_with_q}, the a priori error estimate \eqref{eq:local_approximation_corollary} follows from the weighted Poincar\'e inequality \eqref{eq:wPoincare2}. Note that \eqref{eq:local_approximation_with_q} holds for every $\mathbf{q} \in \mathbb{R}^d$ and that $\mathbf{V} : \mathbb{R}^d \rightarrow \mathbb{R}^d$ is surjective. This concludes the proof.
\end{proof}

\subsection{An error estimate}

We conclude with the following a priori error estimate, which is the main result of this section.

\begin{theorem}[error estimate]
\label{thm:ErrEstEqn}
Let $\Phi$ be an $N$-function satisfying Assumption \ref{ass:equivalence}, and let $\omega \in A_\Phi$. Let $u \in W_{0}^{1,\Phi}(\Omega)$ be the solution of \eqref{eq:nonlinear_weighted_problem_weak}, and let $u_h \in \mathbb{V}_h$ be its finite element approximation, which solves \eqref{eq:nonlinear_weighted_problem_discrete}. If $\nabla \mathbf{V}(\nabla u) \in L^2(\omega,\Omega)$, then
\begin{equation}
\| \mathbf{V}(\nabla u) - \mathbf{V}(\nabla u_h) \|_{L^2(\omega,\Omega)}
\lesssim h
\| \nabla \mathbf{V}(\nabla u) \|_{L^2(\omega,\Omega)},
\label{eq:a_priori_error_bound}
\end{equation}
where the hidden constant depends on $\sigma$, $[\omega]_{A_{i(\Phi)}}$, and the characteristics of $\Phi$.
\end{theorem}
\begin{proof}
The error bound \eqref{eq:a_priori_error_bound} follows from the quasi-best approximation property \eqref{eq:best_approximation} in conjunction with the error estimate \eqref{eq:local_approximation_corollary}.
\end{proof}

\begin{remark}[regularity]
Note that Theorem~\ref{thm:ErrEstEqn} assumes $\GRAD \bV(\GRAD u) \in L^2(\omega,\Omega)$. In the unweighted case, \ie $\omega \equiv 1$, such results are proved in \cite{MR4021898,MR4686656,MR4554061,MR2418205,MR2520895}. 
\end{remark}

\section{The obstacle problem}
\label{sec:Obstacle}

We now consider the constrained minimization of $\calJ$ defined in \eqref{eq:DefOfJ}. In the same setting and with the same assumptions as in Section~\ref{sec:WPhiLap}, we additionally assume that we are given $\psi \in C^2(\bar\Omega)$ with $\psi \leq 0$ on $\partial\Omega$ and define
\begin{equation}
\label{eq:DefOfK}
  \calK_\psi \coloneqq \left\{ w \in W^{1,\Phi}_0(\omega,\Omega) : w \geq \psi \mae \emph{in } \Omega \right\}.
\end{equation}
We recall that $\Phi$ is an $N$-function that satisfies Assumption~\ref{ass:equivalence}, and that the weight $\omega$ belongs to $A_\Phi$. The problem we are interested in is to find $u \in \calK_\psi$ such that
\[
  \calJ(u) \leq \calJ(v) \quad \forall v \in \calK_\psi.
\]
A convex minimization argument \cite{MR4381314,MR881725} shows that the minimizer can be equivalently characterized as the solution of the variational inequality: Find $u \in \calK_\psi$ such that
\begin{equation}
\label{eq:VI}
  \int_\Omega \omega(x) \bA(\GRAD u(x)) \cdot \GRAD(u-v)(x) \diff x \leq \int_\Omega \omega(x) f(x) (u-v)(x) \diff x
  \quad \forall v \in \calK_\psi.
\end{equation}

Since it will be useful for what follows, we define $\lambda \in \left(W^{1,\Phi}(\omega,\Omega)\right)'$ as
\begin{equation}
\label{eq:DefOfLagrangeMult}
  \langle \lambda, w \rangle = \int_\Omega \omega(x) \bA(\GRAD u(x)) \cdot \GRAD w(x) \diff x - \int_\Omega \omega(x) f(x) w(x) \diff x.
\end{equation}
We note that \eqref{eq:VI} can thus be rewritten as follows: Find $u \in \calK_\psi$ such that
\[
  \langle \lambda, u - v \rangle \leq 0 \quad \forall v \in \calK_\psi.
\]

The following result describes further properties of $\lambda$.

\begin{proposition}[properties of $\lambda$]
\label{prop:LagrangeMult}
The functional $\lambda$ defined in \eqref{eq:DefOfLagrangeMult} defines a nonnegative Radon measure. In addition, $\lambda$ also satisfies the following properties:
\begin{itemize}
  \item For every $\phi\in C_0^\infty(\Omega)$ that is nonnegative, \ie $\phi\geq0$ in $\Omega$, we have
  $
    \langle \lambda , \phi \rangle \geq 0.
  $
  
  \item The following \emph{complementarity condition} holds:
  $
    \langle \lambda, u - \psi \rangle = 0.
  $
\end{itemize}
\end{proposition}
\begin{proof}
Let $0 \leq \phi \in C_0^\infty(\Omega)$.
Define $v = u + \phi \in W^{1,\Phi}_0(\omega,\Omega)$. Since $v$ satisfies $v \geq u \geq \psi$ in $\Omega$, $v$ is admissible in \eqref{eq:VI}. We can therefore replace $v$ in the variational inequality \eqref{eq:VI} and obtain
$
  0 \geq \langle \lambda , u -(u+\phi) \rangle = - \langle \lambda , \phi \rangle.
$
This proves the first property. We also note that this shows that $\lambda$ defines a nonnegative distribution. The Riesz-Schwartz theorem (see \cite[Chapter 1,~\S4, Th\'eor\`eme V]{MR209834} and \cite[Chapter 1, Sec.~1.7, Theorem II]{MR2012831}) then shows that $\lambda$ is a nonnegative Radon measure.

To obtain the second property, we define the coincidence and noncoincidence sets
\[
  \mathcal{C} = \{ x \in \Omega : u(x) = \psi(x) \},
  \qquad
  \Omega^+ = \Omega \setminus \mathcal{C} = \{ x \in \Omega : u(x) > \psi(x) \}.
\]
With these sets at hand, we can then write
\begin{equation}
\label{eq:decomposition}
  \langle \lambda , u - \psi \rangle = \int_{\mathcal{C}} (u-\psi)(x) \diff \lambda(x) + \int_{\Omega^+} (u-\psi)(x) \diff \lambda(x) = \int_{\Omega^+} (u-\psi)(x) \diff \lambda(x).
\end{equation}
Now let $\mathcal{N} \subset \Omega^+$ be open, and let $\phi \in C_0^\infty(\mathcal{N})$. Consider $v = u-\epsilon \phi \in W^{1,\Phi}_0(\omega,\Omega)$. For $\epsilon >0$, but sufficiently small, we have that $v  \in \calK_\psi$. In fact, we have that $u(x) > \psi(x)$ up to a null set in $\Omega^+$. Then, for a sufficiently small $\epsilon>0$, we obtain that $u(x) - \epsilon \phi(x) \geq \psi(x)$ in $\Omega$. If we replace this function $v$ in \eqref{eq:VI}, we obtain
\[
  0 \geq \langle \lambda , u - (u - \epsilon \phi) \rangle = \epsilon \langle \lambda, \phi \rangle.
\]
From this and the property that $\langle \lambda , \phi \rangle \geq 0$ for every $\phi\in C_0^\infty(\Omega)$ that is nonnegative, we conclude that $\langle \lambda, \phi \rangle = 0$ for every $\phi \in C_0^\infty(\Omega^+)$ that is nonnegative, \ie
\[
  \int_{\Omega^+} \phi(x) \diff \lambda(x) = 0 \quad \forall \phi \in C_0^\infty(\Omega^+): \ \phi \geq 0 \text{ in }  \Omega^+.
\]
Since $u-\psi > 0$ in $\Omega^+$, an approximation argument and \eqref{eq:decomposition} then show that
$
  \langle \lambda , u - \psi \rangle =  0
$
as claimed.
\end{proof}

\subsection{Finite element discretization}

To approximate the solution of \eqref{eq:VI}, we retain the notation and constructions from Section~\ref{sec:FEM}. In addition, we need to introduce the positivity preserving interpolant $\wp_h : L^1(\Omega) \to \polV_h$, which was originally developed in \cite{MR1742264}. Let $\Vertices = \{ \textv \}$ be the collection of vertices of $\T$ and let $\VerticesIn \coloneqq \Vertices \cap \Omega$. Since $\polV_h$ consists of piecewise linears, there is a bijection between $\VerticesIn$ and the canonical Lagrange basis $\{\phi_\textv\}_{\textv \in \VerticesIn}$ of $\polV_h$. Given $\textv \in \VerticesIn$, we let
\[
  S_{\textv} = \bigcup \left\{ T \in \T: T \ni \textv \right\}.
\]
For $\textv \in \VerticesIn$, we define $\mathcal{B}_\textv$ to be the largest ball centered in $\textv$ and contained in $S_\textv$. The positivity preserving interpolant $\wp_h$ is defined as follows: for $w \in L^1(\Omega)$,
\begin{equation}
\label{eq:DefofChenRhn}
  \wp_h w = \sum_{\textv \in \VerticesIn} \left( \fint_{\mathcal{B}_\textv} w(x) \diff x \right) \phi_\textv.
\end{equation}
Clearly, $w \geq 0$ in $\Omega$ implies $\wp_h w \geq 0$ in $\Omega$. Further properties will be detailed below.

With the interpolant $\wp_h$ at hand, we can define the discrete admissible set
\begin{equation}
\label{eq:DefofKh}
  \calK_{h,\psi} \coloneqq \left\{ w_h \in \polV_h : w_h \geq \wp_h \psi \text{ in } \Omega \right\}.
\end{equation}

The finite element approximation of the solution to \eqref{eq:VI} is $u_h \in \calK_{h,\psi}$, which satisfies the following variational inequality for every $v_h \in \calK_{h,\psi}$: 
\begin{equation}
\label{eq:VIh}
  \int_\Omega \omega(x) \bA(\GRAD u_h(x)) \cdot \GRAD(u_h-v_h)(x) \diff x \leq \int_\Omega \omega(x) f(x) (u_h-v_h)(x) \diff x.
\end{equation}

Standard arguments can be used to prove the existence and uniqueness of a solution of \eqref{eq:VIh}.

\subsection{A priori error bounds}

We now analyze error bounds for our scheme.

\begin{proposition}[first error bound]
\label{prop:FirstBoundObstacle}
Let $u \in \calK_\psi$ and $u_h \in \calK_{h,\psi}$ be the solutions of \eqref{eq:VI} and \eqref{eq:VIh}, respectively. If $i(\Phi) > 2$ and $\lambda \in L^2(\omega^{-1},\Omega)$, then we have
\begin{align*}
  \| \bV(\GRAD u) - \bV(\GRAD u_h) \|_{L^2(\omega,\Omega)}^2 &\lesssim
  \| \lambda \|_{L^2(\omega^{-1},\Omega)} \| \left[ u-\psi \right] - \wp_h\left[ u- \psi\right] \|_{L^2(\omega,\Omega)} \\
    &+ \int_\Omega \omega(x) |\bV(\GRAD u(x)) - \bV(\GRAD \wp_h u(x))|^2 \diff x,
\end{align*}
where $\wp_h$ is the positivity preserving interpolant defined in \eqref{eq:DefofChenRhn}. In the previous estimate, the hidden constant is independent of $h$.
\end{proposition}
\begin{proof}
We first note that since $i(\Phi) > 2$, Proposition~\ref{prop:LPhiweightIntoL2weight} guarantees that $u \in W^{1,2}_0(\omega,\Omega)$. Next, we set some notation. We define the error $e := u-u_h$ and the auxiliary function $z := u - \wp_h u$. We note that $\wp_h u \in \calK_{h,\psi}$.

With the help of \eqref{eq:fundamental_1} we can now derive the following bound:
\[
  \| \bV(\GRAD u) - \bV(\GRAD u_h) \|_{L^2(\omega,\Omega)}^2 \lesssim \int_\Omega \omega(x) \left[ \bA(\GRAD u(x)) - \bA(\GRAD u_h(x)) \right] \cdot \GRAD e(x) \diff x \eqqcolon \mathfrak{D}.
\]
We further decompose $\mathfrak{D}$ as follows:
\begin{align*}
  \mathfrak{D} &= \int_\Omega \omega(x) \left[ \bA(\GRAD u(x)) - \bA(\GRAD u_h(x)) \right] \cdot \GRAD z(x) \diff x \\
    &+ \int_\Omega \omega(x) \left[ \bA(\GRAD u(x)) - \bA(\GRAD u_h(x)) \right] \cdot \GRAD(\wp_h u - u_h)(x) \diff x
    \eqqcolon \mathfrak{L} + \mathfrak{N}.
\end{align*}

Let us first bound the term $\mathfrak{N}$. To do this, we use the discrete variational inequality \eqref{eq:VIh} and obtain
\begin{align*}
  \mathfrak{N} &= \int_\Omega \omega(x) \bA(\GRAD u(x)) \cdot \GRAD(\wp_h u-u_h)(x) \diff x \\
    &+ \int_\Omega \omega(x) \bA(\GRAD u_h(x)) \cdot \GRAD(u_h-\wp_h u)(x) \diff x \\
  &\leq \int_\Omega \omega(x) \bA(\GRAD u(x)) \cdot \GRAD(\wp_h u-u_h)(x) \diff x 
    + \int_\Omega \omega(x) f(x) (u_h - \wp_h u)(x) \diff x \\
  &= \langle \lambda, \wp_h u - u_h \rangle,
\end{align*}
where $\lambda$ is defined in \eqref{eq:DefOfLagrangeMult}. We can then continue and write
\begin{align*}
  \mathfrak{N}
  &\leq
  \langle \lambda, u-\psi \rangle + \langle \lambda , \wp_h \psi - u_h \rangle + \left\langle \lambda, \left[ \wp_h u - \wp_h \psi \right] - \left[ u - \psi \right] \right\rangle \\
  &=
  \langle \lambda , \wp_h \psi - u_h \rangle + \left\langle \lambda, \left[ \wp_h u - \wp_h \psi \right] - \left[ u- \psi \right] \right\rangle,
\end{align*}
where we have used the complementarity condition $\langle \lambda, u-\psi \rangle = 0$ from Proposition~\ref{prop:LagrangeMult}. Let us now examine each of the remaining terms separately.
\begin{itemize}
  \item Since $u_h \in \calK_{h,\psi}$, we have $u_h \geq \wp_h \psi$. This and the fact that $\lambda$ is a nonnegative Radon measure (see Proposition~\ref{prop:LagrangeMult}) allows us to obtain
  \[
    \langle \lambda , \wp_h \psi - u_h \rangle = \int_\Omega (\wp_h \psi - u_h)(x) \diff\lambda(x) \leq 0.
  \]
  
  \item Since we have assumed that $\lambda \in L^2(\omega^{-1},\Omega)$, the fact that $\wp_h$ is linear shows that
  \begin{align*}
    \left\langle \lambda, \left[ \wp_h u- \wp_h \psi \right] - \left[ u- \psi \right] \right\rangle &=
    \left\langle \lambda, \wp_h\left[ u- \psi\right] - \left[ u-\psi \right] \right\rangle \\
      &\leq 
    \| \lambda \|_{L^2(\omega^{-1},\Omega)} \| \left[ u-\psi \right] - \wp_h\left[ u- \psi\right] \|_{L^2(\omega,\Omega)}.
  \end{align*}
  To obtain the last estimate, we used H\"older's inequality.
\end{itemize}
  To summarize, we have come to the following conclusion:
\[
  \mathfrak{N} \leq \| \lambda \|_{L^2(\omega^{-1},\Omega)} \| \left[ u-\psi \right] - \wp_h\left[ u- \psi \right] \|_{L^2(\omega,\Omega)}.
\]

To estimate the term $\mathfrak{L}$, we can repeat the arguments we used to control the term $\mathfrak{I}$ in the proof of Theorem~\ref{thm:CeaLin}. These arguments lead us to the conclusion that
\[
   \mathfrak{L} \leq C \delta \| \bV(\GRAD u) - \bV(\GRAD u_h) \|_{L^2(\omega,\Omega)}^2 + C_\delta \int_\Omega \omega(x) \Phi_{|\GRAD u(x)|}(|\GRAD z(x)|) \diff x.
\]
If we insert the estimates for $\mathfrak{L}$ and $\mathfrak{N}$ into the bound in which $\mathfrak{D}$ is defined, we obtain
 \begin{multline*}
   \| \bV(\GRAD u) - \bV(\GRAD u_h) \|_{L^2(\omega,\Omega)}^2 \leq C \delta \| \bV(\GRAD u) - \bV(\GRAD u_h) \|_{L^2(\omega,\Omega)}^2
   \\
   + C_\delta \int_\Omega \omega(x) \Phi_{|\GRAD u(x)|}(|\GRAD z(x)|) \diff x + C \| \lambda \|_{L^2(\omega^{-1},\Omega)} \| \left[ u-\psi \right] - \wp_h\left[ u- \psi \right] \|_{L^2(\omega,\Omega)}.
\end{multline*}
The term $\| \bV(\GRAD u) - \bV(\GRAD u_h) \|_{L^2(\omega,\Omega)}^2$ appearing on the right hand side of the previous bound can be absorbed on the left hand side if $\delta$ is chosen carefully. This shows that
\begin{multline*}
  \| \bV(\GRAD u) - \bV(\GRAD u_h) \|_{L^2(\omega,\Omega)}^2 \lesssim  \| \lambda \|_{L^2(\omega^{-1},\Omega)} \| \left[ u- \psi \right] - \wp_h\left[ u- \psi \right] \|_{L^2(\omega,\Omega)}
  \\
  + \int_\Omega \omega(x) \Phi_{|\GRAD u(x)|}(|\GRAD z(x)|) \diff x, \qquad z = u - \wp_h u.
\end{multline*}
Use the equivalences \eqref{eq:fundamental_1}--\eqref{eq:fundamental_2} one last time to obtain the assertion and conclude.
\end{proof}

\subsection{The positivity preserving interpolant}
\label{sub:Crhn}

Note that up to this point no property other than positivity has been required of $\wp_h$. In the following, however, the stability and approximation properties of $\wp_h$ are required. Ideally, we would like to have an operator that is a projection and satisfies \eqref{eq:stability_SZ}. Note, however, that $\wp_h$ is not a projection. In fact, it is not possible to construct a positivity preserving projection; see \cite{MR1933037}. For this reason, we need to prove suitable properties for $\wp_h$.

\begin{lemma}[properties of $\wp_h$]
\label{lem:propChenRhn}
The interpolant $\wp_h$ defined in \eqref{eq:DefofChenRhn} satisfies the following properties.
\begin{enumerate}
  \item $\wp_h$ is positivity preserving: If $w \geq 0$ in $\Omega$, then $\wp_h w\geq 0$ in $\Omega$.
  
  \item $\wp_h$ is symmetric in the sense that: If $\textv \in \VerticesIn$ and $w_{|\mathcal{B}_\textv} \in \polP_1$, then $\wp_h w(\textv) = w(\textv)$.
  
  \item $\wp_h$ is locally invariant: For every $T \in \T$ with $T \cap \partial \Omega = \emptyset$, if $w_{|S_T} \in \polP_1$, then $\wp_h w_{|T}= w_{|T}$.
  
  \item $\wp_h$ is locally stable: For all $p \in [1,\infty]$ and every $T \in \T$, we have
  \begin{align*}
    \| \wp_h w \|_{L^p(T)} &\lesssim \| w \|_{L^p(S_T)} \quad \forall w \in L^p(\Omega), \\
    \| \GRAD \wp_h w \|_{L^p(T)} &\lesssim \| \GRAD w \|_{L^p(S_T)} \quad \forall w \in W^{1,p}_0(\Omega).
  \end{align*}
  
  \item $\wp_h$ is weighted Orlicz stable: Let $\Phi$ be an $N$-function that satisfies Assumption~\ref{ass:equivalence}, and let $\omega \in A_{\Phi}$. Then, for every $T \in \T$ and for all $a \geq 0$,
  \[
    \fint_T \omega(x) \Phi_a(|\GRAD \wp_h w(x)|) \diff x \lesssim \fint_{S_T} \omega(x) \Phi_a(|\GRAD w(x)|) \diff x,
  \]
  where the hidden constant does not depend on $T$ or $a$.
  
  \item $\wp_h$ has weighted approximation properties: If $\omega \in A_2$ and $w \in W^{1,2}_0(\omega,\Omega)$, then, for all $T \in \T$,
  \[
    \| w - \wp_h w \|_{L^2(\omega,T)} \lesssim h_T \| \GRAD w \|_{L^2(\omega,S_T)}.
  \]
\end{enumerate}
\end{lemma}
\begin{proof}
We examine each statement individually.
\begin{enumerate}
  \item This property follows directly from the definition of $\wp_h$.
  
  \item Since $w_{|\mathcal{B}_\textv} \in \polP_1$, we write $w(x) = a + \bb\cdot(x-\textv)$ for $a \in \R$ and $\bb \in \R^d$. We have
  \[
    \wp_h w(\textv) = \fint_{\mathcal{B}_\textv} \left[ a + \bb\cdot(x-\textv) \right] \diff x = a + \bb \cdot \fint_{\mathcal{B}_\textv} (x-\textv) \diff x = a = w(\textv).
  \]

  \item This follows from the repetition of the previous calculation for each vertex of $T$, which forms an unisolvent set for $\polP_1(T)$.

  \item This is proved in \cite[Lemma 3.1]{MR1742264}.
  
  \item This follows from the preceding stability properties by repeating the arguments in the proofs of Theorem~\ref{thm:WStabSZ} and Corollary~\ref{cor:WStabSZ}.
  
  \item Let $T \in \T$, and let $w \in W^{1,2}_0(\omega,\Omega)$. Let $\frakp \in \polP_1(\R^d)$ be arbitrary. We begin the proof by noting that if we can show the bound
  \begin{equation}
  \label{eq:BrambleHilbert}
    \| \wp_h w- \frakp \|_{L^2(\omega,T)} \lesssim \| w - \frakp \|_{L^2(\omega,S_T)} + h_T \| \GRAD( w - \frakp )\|_{L^2(\omega,S_T)},
  \end{equation}
  then a simple application of the triangle inequality shows
  \[
    \| w-\wp_h w\|_{L^2(\omega,T)} \lesssim \| w - \frakp \|_{L^2(\omega,S_T)} + h_T \| \GRAD( w - \frakp )\|_{L^2(\omega,S_T)}.
  \]
  A suitable choice of $\frakp$ shall then imply the thesis.
  
  To prove \eqref{eq:BrambleHilbert}, we proceed differently depending on whether $T$ is an interior element or $T$ touches the boundary. If $T$ is an interior element, \ie $T \cap \partial \Omega = \emptyset$, then local invariance and stability in $L^2(\omega,T)$ imply the result. On the other hand, if the element $T$ touches the boundary we proceed as follows. Let $\{\textv_i^T \}_{i=0}^d$ be the collection of vertices of $T$. Since $(\wp_h w - \frakp)_{|T} \in \polP_1$, we can write
  \[
    (\wp_h w - \frakp)_{|T} = \sum_{i=0}^d (\wp_h w - \frakp)(\textv_i^T) \phi_{\textv_i^T}.
  \]
  We now use that $\| \phi_\textv \|_{L^\infty(\Omega)} = 1$ for all $\textv \in \Vertices$ to obtain the estimate
  \[
    \| \wp_h w - \frakp \|_{L^2(\omega,T)}^2 \lesssim \max_{i=0}^d \left| (\wp_h w - \frakp)(\textv_i^T) \right|^2 \omega(T).
  \]
  Now we distinguish two cases. If $\textv_i^T \in \VerticesIn$, then
  \begin{align*}
    \left| (\wp_h w - \frakp)(\textv_i^T) \right| &= \left| \wp_h (w - \frakp)(\textv_i^T) \right| \leq \fint_{\mathcal{B}_{\textv_i^T}} | w(x) - \frakp(x)| \diff x \\
    &\lesssim h_T^{-d} \omega^{-1}(T)^{\frac{1}{2}} \| w - \frakp \|_{L^2(\omega,S_T)}.
  \end{align*}
  On the other hand, if $\textv_i^T \in \partial \Omega$, then $\wp_h w(\textv_i^T) = 0$ by definition. Let $F \subset T$ be the $(d-1)$-dimensional subsimplex containing $\textv_i^T$ such that $F \subset \partial \Omega$. We thus have
  \[
    \left| (\wp_h w - \frakp)(\textv_i^T) \right| = |\frakp(\textv_i^T)| \lesssim \fint_{F} |\frakp(x)| \diff x = \fint_{F} |w(x)- \frakp(x)| \diff x,
  \]
  where we have used that $w_{|F} = 0$. The scaled trace inequality \eqref{eq:ScaledTrace} then yields
  \begin{multline*}
    \fint_{F} |w(x)- \frakp(x)| \diff x \lesssim \fint_{S_T} | w(x) - \frakp(x)| \diff x + h_T \fint_{S_T} |\GRAD( w - \frakp)(x)| \diff x \\
    \leq h_T^{-d} \omega^{-1}(S_T)^{\frac{1}{2}} \left[ \| w - \frakp \|_{L^2(\omega,S_T)} + h_T \| \GRAD (w-\frakp)\|_{L^2(\omega,S_T)} \right].
  \end{multline*}

To summarize, we have arrived at the following estimate:
  \begin{align*}
    \| \wp_h w - \frakp \|_{L^2(\omega,T)}^2 &\lesssim \frac{\omega(T)\omega^{-1}(S_T)}{h_T^{2d}} \left[ \| w - \frakp \|_{L^2(\omega,S_T)} + h_T \| \GRAD (w-\frakp)\|_{L^2(\omega,S_T)} \right]^2 \\
    &\lesssim [\omega]_{A_2} \left[ \| w - \frakp \|_{L^2(\omega,S_T)} + h_T \| \GRAD (w-\frakp)\|_{L^2(\omega,S_T)} \right]^2.
  \end{align*}
\end{enumerate}

All properties have been proved. This concludes the proof.
\end{proof}

Next, we obtain an analogue of the bound \eqref{eq:local_approximation}.

\begin{proposition}[interpolation error estimate]
\label{prop:InterpObstacle}
Let $\Phi$ be an $N$-function that satisfies Assumption~\ref{ass:equivalence}, and let $\omega \in A_{i(\Phi)}(\Omega)$. Let us assume that $\T = \{ T \}$ is such that no simplex $T$ has more than one $(d-1)$-dimensional subsimplex on $\partial \Omega$. If $w \in W^{1,\Phi}_0(\omega,\Omega)$ is such that $\GRAD \bV(\GRAD w ) \in L^2(\omega,\Omega)$ and $h$ is sufficiently small, then
\[
  \|\bV(\GRAD w) - \bV(\GRAD \wp_h w) \|_{L^2(\omega,\Omega)}  \lesssim h \| \GRAD \bV(\GRAD w) \|_{L^2(\omega,\Omega)},
\]
where the constant depends on $[\omega]_{A_{i(\Phi)}}$ and $\|\omega\|_{L^\infty(\calG)}$, but is independent of $h$.
\end{proposition}
\begin{proof}
We begin by noting that for every $T \in \T$ and every $\frakp \in \polP_1(\R^d)$, we have
\begin{align*}
  \int_T \omega(x) |\bV(\GRAD w(x)) - \bV(\GRAD \wp_h w) |^2 \diff x &\lesssim \int_T \omega(x) |\bV(\GRAD w(x)) - \bV(\GRAD \frakp) |^2 \diff x \\
  &+ \int_T \omega(x) |\bV(\GRAD \frakp) - \bV(\GRAD \wp_h w) |^2 \diff x \eqqcolon \bI + \bI\bI.
\end{align*}

Define the sets $
  \T^{\mathrm{in}} \coloneqq \left\{ T \in \T: T \cap \partial\Omega = \emptyset \right\}
$
and
$
  \T^{\partial} \coloneqq \left\{ T \in \T: T \cap \partial\Omega \neq \emptyset \right\}
$. 
We now partition $\T = \T^{\mathrm{in}} \sqcup \T^{\partial}$.

Let $T \in \T^{\mathrm{in}}$. The local invariance of $\wp_h$, proved in Lemma~\ref{lem:propChenRhn}, implies that $\wp_h \frakp_{|T} = \frakp_{|T}$. Using the equivalences \eqref{eq:fundamental_1}--\eqref{eq:fundamental_2} and the weighted Orlicz stability of $\wp_h$, which was derived in Lemma~\ref{lem:propChenRhn}, we obtain the estimate
\begin{align*}
  \bI\bI &\lesssim \int_T \omega(x) \Phi_{|\GRAD \frakp|}(|\GRAD \wp_h (w - \frakp)(x)| ) \diff x
    \lesssim \int_{S_T} \omega(x) \Phi_{|\GRAD \frakp|}(|\GRAD(w - \frakp)(x)| \diff x \\
    &\lesssim \int_{S_T} \omega(x) |\bV(\GRAD w(x)) - \bV(\GRAD \frakp) |^2 \diff x.
\end{align*}
It follows that for every $T \in \T^{\mathrm{in}}$ the following bound applies:
\[
  \int_T \omega(x) |\bV(\GRAD w(x)) - \bV(\GRAD \wp_h w) |^2 \diff x \lesssim \int_{S_T} \omega(x) |\bV(\GRAD w(x)) - \bV(\GRAD \frakp) |^2 \diff x,
\]
where $\frakp \in \polP_1(\R^d)$ is arbitrary. Let $\bQ \in \mathbb{R}^d$ and choose $\bz = \GRAD \frakp \in \R^d$, so that $\bV(\bz) = \bQ$.
This is possible because $\bV$ is surjective. This gives us for every $T \in \T^{\mathrm{in}}$ 
\begin{equation}
\label{eq:InterpObstacleInterior}
  \begin{aligned}
    \int_T \omega(x) |\bV(\GRAD w(x)) - \bV(\GRAD \wp_h w) |^2 \diff x &\lesssim \int_{S_T} \omega(x) |\bV(\GRAD w(x)) - \bQ |^2 \diff x \\
    &\lesssim h_T^2 \int_{S_T} \omega(x) |\GRAD \bV(\GRAD w(x))|^2 \diff x,
  \end{aligned}
\end{equation}
where we used the weighted Poincar\'e inequality \eqref{eq:wPoincare2} in the last step, by choosing the vector $\bQ$ accordingly, and that by shape regularity $\diam(T) \simeq \diam(S_T)$.

We now consider $T \in \T^\partial$. If $h$ is sufficiently small, it is not a mistake to assume that $T \subset \calG$. Recall that $\calG$ is as in Definition \ref{def:A_p(Omega)}. Let $F \subset T$ be the unique $(d-1)$-dimensional subsimplex such that $F = T \cap \partial\Omega$ and denote by $\bzeta$ the unit outer normal to $T$ on $F$. Now let $\{\textv_i^T\}_{i=0}^d$ be the set of vertices of $T$ and assume that $\{\textv_i^T\}_{i=0}^d$ is numbered such that $F = \conv\{\textv_i^T\}_{i=0}^{d-1}$. Define $\frakp \in \polP_1(\R^d)$ as $\frakp(x) = b \bzeta\cdot(x-\textv_0^T)$, where $b \in \R$ is to be chosen. We note that, by construction, $\frakp_{|F} = 0$. Moreover, we have by definition that $\wp_h\frakp_{|F} = 0$. Finally, the symmetry of $\wp_h$ proved in Lemma~\ref{lem:propChenRhn} implies that $\wp_h \frakp(\textv_d^T) = \frakp(\textv_d^T)$. From this follows that $\wp_h \frakp_{|T} = \frakp_{|T}$.

Based on the constructions described above, the weighted Orlicz stability of $\wp_h$ (see Lemma~\ref{lem:propChenRhn}) allows us to treat $\bI\bI$ as before and obtain that, for every $T \in \T^\partial$,
\[
  \int_T \omega(x) |\bV(\GRAD w(x)) - \bV(\GRAD \wp_h w) |^2 \diff x \lesssim \int_{S_T} \omega(x) |\bV(\GRAD w(x)) - \bV(\GRAD \frakp) |^2 \diff x,
\]
where $\frakp \in \polP_1(\R^d)$ is such that $\frakp_{|F} =0$, but is otherwise arbitrary.

Let us now use the fact that $w_{|F}=0$ to deduce that $\GRAD w$ and $\bzeta$ are parallel on $F$. This implies that $\bV(\GRAD w(x))_{|F} = v(x) \bzeta$ for some $v: F \to \R$. Now, given $x \in T$, we denote by $\hat x \in F$ the projection of $x$ onto $F$. We note that, due to shape regularity, $|\textv_d^T - \hat{\textv}_d^T| \simeq h_T$. Define $Q$ to be the prism with base $F$ and height $H \coloneqq |\textv_d^T - \hat{\textv}_d^T|$. We may now extend $v$ to $Q$ via
\[
  v(x) = v(\hat x) \quad \forall x \in Q.
\]
As a final preparatory step, we note that $\bV(\GRAD \frakp) = \kappa \bzeta$, where $\kappa \in\R$ can be chosen arbitrarily by suitably specifying the value $\frakp(\textv_d^T)$. We now estimate
\begin{align*}
  \bI &\lesssim \int_T \omega(x) |\bV(\GRAD w(x)) - v(x)\bzeta|^2 \diff x + \int_T \omega(x) |v(x) - \kappa|^2 \diff x \eqqcolon \bI_1 + \bI_2.
\end{align*}

In order to bound $\bI_1$, we set some notation. We let $\bw(x) = \bV(\GRAD w(x))$. Then,
\begin{align*}
  \bI_1 &\leq \int_Q \omega(x) |\bw(x)-\bw(\hat x)|^2 \diff x
    =\int_Q \omega(x)  \left|\int_0^1\GRAD \bw(tx + (1-t)\hat x) \cdot (x - \hat x) \diff t  \right|^2 \diff x \\
    & \lesssim \| \omega \|_{L^\infty(\calG)}\int_Q \int_0^1 \left|\GRAD \bw(tx + (1-t)\hat x) \cdot (x - \hat x)\right|^2 \diff t \diff x,
\end{align*}
where we have used the fact that $\omega$ belongs to the reduced class $A_{i(\Phi)}(\Omega)$ (see Definition \ref{def:A_p(Omega)}). We now introduce a coordinate system in which the first axis is aligned with $\bzeta$ and $Q$ is on the half space defined by $\{ x\in\R^d : x_1 \geq 0 \}$. We thus have that, if $x \in Q$, then $x = (x_1, x')^\top$ with $x_1 \geq 0$ and $x' \in F \subset \R^{d-1}$. We also note that $\hat{x} = (0,x')^\top$, $x-\hat x = (x_1,0)^\top$, and $y= tx + (1-t)\hat x = (tx_1,x')^\top$ (see Figure~\ref{fig:guatilla}). With these observations in mind, we continue the bound for the term $\bI_1$ as follows:
\begin{align*}
  \bI_1 &\lesssim \| \omega \|_{L^\infty(\calG)}\int_Q \int_0^1 |\partial_1 \bw(tx_1,x')|^2 |x_1|^2 \diff t \diff x \\
    &= \| \omega \|_{L^\infty(\calG)} \int_0^H x_1^2 \int_F \int_0^1 |\partial_1 \bw(tx_1,x')|^2 \diff t \diff x' \diff x_1 \\
    &= \| \omega \|_{L^\infty(\calG)} \int_0^H x_1 \int_F \int_0^{x_1} |\partial_1 \bw(r,x')|^2 \diff r \diff x' \diff x_1 \\
    &\leq \| \omega \|_{L^\infty(\calG)} \int_0^H x_1 \int_F \int_0^H |\partial_1 \bw(r,x')|^2 \diff r \diff x' \diff x_1,
\end{align*}
where we have introduced the change of variables $r = tx_1$ in the innermost integral and used the fact that if $x \in Q$, then $x_1 \in [0,H]$. Now we just integrate with respect to $x_1$ and use that $H \simeq h_T$ to obtain
\begin{align*}
  \bI_1 &\lesssim \frac12 h_T^2 \| \omega \|_{L^\infty(\calG)} \int_F \int_0^H |\partial_1 \bw(r,x')|^2 \diff r \diff x' \\
    &\leq \frac12 h_T^2 \| \omega \|_{L^\infty(\calG)} \int_F \int_0^H |\GRAD \bw(r,x')|^2 \diff r \diff x' \\
    &= \frac12 h_T^2 \| \omega \|_{L^\infty(\calG)} \int_Q |\GRAD \bw(y)|^2 \diff y \\
    &\leq \frac12 h_T^2 \| \omega \|_{L^\infty(\calG)} \| \omega^{-1} \|_{L^\infty(Q)} \int_Q \omega(y ) |\GRAD \bV(\GRAD w(y ))|^2 \diff y.
\end{align*}
We recall that $Q = [0,H] \times F$ and $\bw(y) = \bV(\GRAD \bw(y))$.

To bound $\bI_2$, we recall that $H = |\textv_d^T - \hat{\textv}_d^T|$ and obtain that
\begin{align*}
  \bI_2 &\leq \| \omega \|_{L^\infty(\calG)} \int_0^H \int_F |v(\hat x) - \kappa |^2 \diff \hat x \diff t
  \lesssim h_T \| \omega \|_{L^\infty(\calG)} \int_F | \tr_F \bw(\hat x) - \kappa \bzeta |^2 \diff \hat x \\
    &\lesssim h_T \| \omega \|_{L^\infty(\calG)} \left[ h_T^{-1} \int_T |\bw(x) - \kappa \bzeta|^2 \diff x + h_T \int_T |\GRAD( \bw(x) - \kappa\bzeta)|^2 \diff x \right] \\
    & \lesssim \| \omega \|_{L^\infty(\calG)} \| \omega^{-1} \|_{L^\infty(T)}
      \left[ \int_T \omega(x)|\bw(x) - \kappa \bzeta|^2 \diff x + h_T^2 \int_T \omega(x) |\GRAD \bw(x)|^2 \diff x \right],
\end{align*}
where we have used the scaled trace inequality \eqref{eq:ScaledTrace}. The weighted Poincar\'e inequality then shows that for every $T \in \T^\partial$,
\begin{equation}
\label{eq:InterpObstacleBoundary}
  \int_T \omega(x) |\bV(\GRAD w(x)) - \bV(\GRAD \wp_h w) |^2 \diff x 
    \lesssim h_T^2 \int_{S_T} \omega(x) |\GRAD \bV(\GRAD w(x))|^2 \diff x,
\end{equation}
where the hidden constant depends on $[\omega]_{A_{i(\Phi)}}$ and $\|\omega\|_{L^\infty(\calG)}$.

It remains to add \eqref{eq:InterpObstacleInterior} and \eqref{eq:InterpObstacleBoundary} and use shape regularity to obtain the claim.
\end{proof}

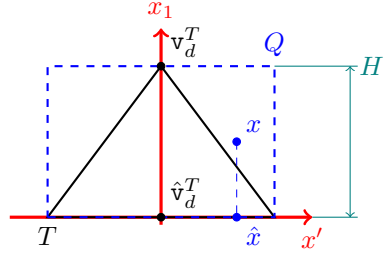
\begin{figure}
  \begin{center}
    \begin{tikzpicture}
      \draw[very thick, red, ->] (0,-0.1) -- (0,2.5) node[above]{$x_1$};
      \draw[very thick, red, ->] (-2,0) -- (2,0) node[below]{$x'$};

      \draw[thick, black] (-1.5,0) -- (1.5,0) -- (0, 2) -- cycle node[below] {$T$};

      \draw[thick, dashed, blue] (-1.5,0) -- (1.5,0) -- (1.5, 2) node[above] {$Q$} -- (-1.5,2) -- cycle;

      \draw[fill=black,black] (0,2) circle[radius=0.05] node[above right, black] {$\textv_d^T$};
      \draw[fill=black,black] (0,0) circle[radius=0.05] node[above right, black] {$\hat{\textv}_d^T$};

      \draw[fill=blue,blue] (1,1) circle[radius=0.05] node[above right, blue] {$x$};
      \draw[fill=blue,blue] (1,0) circle[radius=0.05] node[below right, blue] {$\hat x$};
      \draw[very thin, dashed, blue] (1,1) -- (1,0) ;

      \draw[teal] (2,0) -- (2.6,0);
      \draw[teal] (1.5,2) -- (2.6,2);
      \draw[teal, <->] (2.5,0) -- (2.5,2) node[right, teal] {$H$};
    \end{tikzpicture}
  \end{center}
  \caption{The geometric setting for the bound for $\mathbf{I}_1$ in the proof of Proposition~\ref{prop:InterpObstacle}. The element $T \in \T^\partial$ is shown in \textbf{solid black}, while the related prism is in \textcolor{blue}{\textbf{dashed blue}}. We also show a generic point $x \in Q$ and its projection onto $F$, \ie $\hat x \in F$.}
\label{fig:guatilla}
\end{figure}

\subsection{An error estimate}

We conclude with the following error estimate for the obstacle problem. Recall that we have assumed that the obstacle satisfies $\psi \in C^2(\bar\Omega)$.

\begin{theorem}[error estimate]
Let $\Phi$ be an $N$-function with $i(\Phi) > 2$ that satisfies Assumption~\ref{ass:equivalence}, and let $\omega \in A_2(\Omega)$. Let $u \in W^{1,\Phi}_0(\omega,\Omega)$ and $u_h \in \polV_h$ solve \eqref{eq:VI} and \eqref{eq:VIh}, respectively. Assume that $\bV(\GRAD u) \in W^{1,2}(\omega,\Omega)$, and that $\lambda$, defined in \eqref{eq:DefOfLagrangeMult}, belongs to $L^2(\omega^{-1},\Omega)$. If $\T$ is such that no simplex has more than one $(d-1)$-dimensional subsimplex on $\partial \Omega$ and $h$ is sufficiently small, then
\[
  \| \bV(\GRAD u) - \bV(\GRAD u_h) \|_{L^2(\omega,\Omega)} \lesssim h \| \GRAD \bV(\GRAD u) \|_{L^2(\omega,\Omega)} + h^{\frac{1}{2}} \| \GRAD (u - \psi) \|_{L^2(\omega,\Omega)},
\]
where the hidden constant is independent of $h$.
\end{theorem}
\begin{proof}
We begin with the observation that, since $i(\Phi) > 2$, we have that $A_2 \subset A_{i(\Phi)}$, and thus that $\omega \in A_{i(\Phi)}(\Omega)$. We may thus apply Proposition~\ref{prop:InterpObstacle} to the conclusion of Proposition~\ref{prop:FirstBoundObstacle} and obtain
\[
  \| \bV(\GRAD u) - \bV(\GRAD u_h )\|_{L^2(\omega,\Omega)}^2 \lesssim h^2 \| \GRAD \bV(\GRAD u) \|_{L^2(\omega,\Omega)}^2 + \| [u-\psi] - \wp_h[u-\psi] \|_{L^2(\omega,\Omega)}.
\]
It remains to use the approximation property of the positivity preserving interpolant $\wp_h$ proved in Lemma~\ref{lem:propChenRhn} to be able to conclude.
\end{proof}

\begin{remark}[regularity]
Note that the error estimate depends on the assumption that $\bV(\GRAD u) \in W^{1,2}(\omega,\Omega)$. In addition, we need the regularity assumption $\lambda \in L^2(\omega^{-1},\Omega)$. We leave the exploration of these properties as open conjectures and refer the reader to \cite{MR4487747,MR4653772,MR3829549,MR4503149,MR4649160,MR4396691} for some results in this direction.
\end{remark}

\bibliographystyle{siamplain}
\bibliography{biblio}

\end{document}